\documentclass[a4,twoside, 10pt]{article}

\usepackage[a4paper, textwidth=15.75 cm, textheight=23.4cm, marginratio={5:5,5:6},marginparwidth=2.1cm]{geometry}

\newcommand{\myauthor}{Bill Deng, Mircea Voineagu}

\newcommand{\mytitle}{Bredon motivic cohomology of the real numbers}

\title{\mytitle}
\author{  
Bill Deng,
Mircea Voineagu}
\date{}

\usepackage{amsmath}
\usepackage{amscd}
\usepackage{amsthm} 
\usepackage{amssymb}
\usepackage{amsfonts}
\usepackage{eucal}
\usepackage{microtype}
\usepackage{mathrsfs}
\usepackage{dsfont}
\usepackage{stmaryrd}
\usepackage[margin = 1.5cm]{caption}
\usepackage[shortlabels]{enumitem}
\setlist[enumerate]{font=\upshape, topsep=0.5pt,itemsep=-0.5ex,partopsep=1ex,parsep=1ex} 
\setlist[enumerate,1]{label = (\arabic*)}
\setlist[1]{labelindent=\parindent} 
\setlist[itemize,1]{topsep=0.5pt,itemsep=-0.5ex,partopsep=1ex,parsep=1ex} 

\usepackage{fancyhdr}
\usepackage[dvipsnames]{xcolor} 
\definecolor{refkey}{gray}{0.5}
\definecolor{labelkey}{gray}{0.5}
\definecolor{note}{rgb}{0.94, 0.99, 1.00}
\colorlet{myurlcolor}{Aquamarine}
\colorlet{mylinkcolor}{violet}
\colorlet{mycitecolor}{YellowOrange}
\definecolor{reference}{rgb}{.93,.51,.93}
\definecolor{citation}{rgb}{1,.68,.26}
\definecolor{mrnumber}{rgb}{.80,.40,0}
\newcommand\myshade{85}

\usepackage{graphicx}
\usepackage{float}
\usepackage[all]{xy} 
\usepackage{tikz}
\usetikzlibrary{patterns}
\tikzstyle{mygrid}=[gray!25!white]

\tikzstyle{myfill} = [fill=gray, fill opacity= 0.2]
\tikzstyle{myhatch} = [pattern=north west lines, pattern color=blue!35]
\tikzstyle{myfill1} = [fill=green!70!white, fill opacity= 0.1]
\tikzstyle{myfill2} = [fill=blue!70!white, fill opacity= 0.1]
\tikzstyle{myfill3} = [fill=red!70!white, fill opacity= 0.1]
\tikzstyle{myline} = [gray!40]
\usetikzlibrary{cd}
\usetikzlibrary{plotmarks}
\usetikzlibrary{arrows}
\makeatletter                                    
\g@addto@macro\@floatboxreset\centering          
\makeatother                                     
\usepackage{subcaption}

\usepackage[pdfstartview=FitH,
pdfauthor={\myauthor},
pdftitle={\mytitle},
colorlinks,
linkcolor=mylinkcolor!\myshade!black,
citecolor=citation!\myshade!black,
urlcolor=myurlcolor!\myshade!black,
backref=page,
bookmarks, 
bookmarksopen=true, 
bookmarksopenlevel=3, 
bookmarksdepth=3]{hyperref}

\usepackage[textsize=scriptsize, color=note, linecolor=lightgray]{todonotes}

\usepackage{aliascnt}

\newcommand{\C}{\mathds{C}} 
\newcommand{\A}{\mathds{A}}
\renewcommand{\P}{\mathds{P}}
\newcommand{\Z}{\mathds{Z}}
\newcommand{\R}{\mathds{R}}

\newcommand{\rH}{\widetilde{H}}
\newcommand{\Br}{{Br}}

\newcommand{\pt}{{pt}}

\renewcommand{\Re}{\mathrm{Re}}

\newcommand{\Sm}{\mathrm{Sm}}

\newcommand{\SH}{\mathrm{SH}}

\newcommand{\EG}{\mathbf{E}}
\newcommand{\BG}{\mathbf{B}}
\newcommand{\EGt}{\widetilde{\EG} }

\newcommand{\wt}[1]{\widetilde{#1}}

\newcommand{\ul}[1]{\underline{\smash{#1}}}

\newcommand{\al}{\alpha}

\newcommand{\hr}{\tilde{H}}
\newcommand{\susp}{\Sigma}

\newcommand{\Si}{\Sigma}
\newcommand{\sd}{_{\Si_2}}

\newcommand{\tl}{\tilde}
\newcommand{\dl}{\delta}

\newcommand{\te}{\tl{E}}
\newcommand{\et}{{\te\sd\ct}}

\newcommand{\ct}{{C_2}}

\newcommand{\fc}{\frac}
\newcommand{\ta}{\theta}

\newcommand{\txo}{\quad\text{or}\quad}
\newcommand{\txa}{\quad\text{and}\quad}

\newcommand{\ka}{\kappa}

\newcommand{\io}{\iota}

\newcommand{\cp}{\cdot}

\newcommand{\kc}[1]{\ka_2^{-#1}}
\newcommand{\ek}{{E_{\Sigma_2}C_2}}

\newcommand{\xto}[1]{\xrightarrow{#1}}

\newcommand{\st}{\star}

\newcommand{\MMt}{\mathds{M}^{C_2}_2} 
\newcommand{\MMtn}{\mathds{M}^{C_2}_n}

\DeclareMathOperator*{\colim}{\mathrm{colim}}

\DeclareMathOperator{\spec}{\mathrm{Spec}}

\newcommand{\cd}{\smash\cdot}

\makeatletter
\newcommand{\subalign}[1]{%
	\vcenter{%
		\Let@ \restore@math@cr \default@tag
		\baselineskip\fontdimen10 \scriptfont\tw@
		\advance\baselineskip\fontdimen12 \scriptfont\tw@
		\lineskip\thr@@\fontdimen8 \scriptfont\thr@@
		\lineskiplimit\lineskip
		\ialign{\hfil$\m@th\scriptstyle##$&$\m@th\scriptstyle{}##$\hfil\crcr
			#1\crcr
		}%
	}%
}
\makeatother

\numberwithin{equation}{section} 

\theoremstyle{plain} 
 
\newaliascnt{theorem}{equation}  
\newtheorem{theorem}[theorem]{Theorem}  
\aliascntresetthe{theorem}

\newaliascnt{proposition}{equation}  
\newtheorem{proposition}[proposition]{Proposition}
\aliascntresetthe{proposition}

\newaliascnt{lemma}{equation}    
\newtheorem{lemma}[lemma]{Lemma}
\aliascntresetthe{lemma}

\newaliascnt{corollary}{equation}  
\newtheorem{corollary}[corollary]{Corollary}
\aliascntresetthe{corollary}

\newaliascnt{claim}{equation}  

\aliascntresetthe{claim}

 \theoremstyle{definition}

\newaliascnt{definition}{equation}  
\newtheorem{definition}[definition]{Definition}
\aliascntresetthe{definition}

\newaliascnt{example}{equation}  

\aliascntresetthe{example}

\newaliascnt{remark}{equation}   
\newtheorem{remark}[remark]{Remark}
\aliascntresetthe{remark}

\newaliascnt{condition}{equation}

\aliascntresetthe{condition}

\newaliascnt{notationconvention}{equation}

\aliascntresetthe{notationconvention}

\newaliascnt{notation}{equation}

\aliascntresetthe{notation}

\begin{document} 
  \maketitle
  \begin{abstract}
Over the real numbers with $\Z/2-$coefficients, we compute the $C_2$-equivariant Borel motivic cohomology ring, the Bredon motivic cohomology groups and prove that the Bredon motivic cohomology ring of real numbers is a proper subring in the $RO(C_2\times C_2)$-graded Bredon cohomology ring of a point. 

This generalizes Voevodsky's computation of the motivic cohomology ring of the real numbers to the $C_2$-equivariant setting. These computations are extended afterwards to any real closed field.

   
\paragraph{Mathematics Subject Classification 2010.}
Primary:
    \href{https://mathscinet.ams.org/msc/msc2010.html?t=14Fxx&btn=Current}{14F42},
    \href{https://mathscinet.ams.org/msc/msc2010.html?t=55Pxx&btn=Current}{55P91}.
Secondary:
    \href{https://mathscinet.ams.org/msc/msc2010.html?t=55Pxx&btn=Current}{55P42},
    \href{https://mathscinet.ams.org/msc/msc2010.html?t=55Pxx&btn=Current}{55P92}.
    
\end{abstract} 

  \tableofcontents
  
\section{Introduction}

 A fundamental principle of modern homotopy theory is that group actions on homotopical objects reveal interesting and otherwise hard to find information about the underlying homotopical objects. Given the importance of motivic stable homotopy theory and its relations with equivariant stable homotopy theory (see for example \cite{BS}, \cite{Bac}), it is therefore important to study equivariant motivic stable homotopy theory (see for example \cite{Hoy1}, \cite{HOV1}, \cite{Del} \cite{HKO}). Bredon motivic cohomology, given by the equivariant study of Voevodsky's motivic cohomology spectrum, was introduced in \cite{HOV} and \cite{HOV1}, and belongs to a larger group of $C_2$-motivic invariants, such as Hermitian K-theory or motivic real cobordism. Bredon motivic cohomology also appears as the zero slice of the equivariant motivic sphere \cite{OH}.
 
 Concrete computations in Bredon motivic cohomology are essential for applications of the theory to other motivic and topological invariants. In many cases, these computations shed new light on the well-known computations of classical motivic cohomology, and, as in the case of Bredon cohomology, they are more difficult and contain more information about the underlying object, even in the case of a trivial $C_2$-action.
 
 In \cite{HOV2}, the second author together with J.Heller and P.A.{\O}stv{\ae}r  computed completely the Bredon motivic cohomology rings of the complex numbers  and of $\EG C_2$ (over the complex numbers). In this paper, we compute the Bredon motivic cohomology of the real numbers and of $\EG C_2$ (over the real numbers). In particular, we generalize the classical motivic cohomology computation of Voevodsky of the motivic cohomology of the real numbers by computing the Bredon motivic cohomology groups of real numbers and by showing that the Bredon motivic cohomology ring of the real numbers is a subring of the $RO(C_2\times\Sigma_2)$-graded Bredon cohomology of a point. Moreover, we show that the Borel motivic cohomology ring of the real numbers is a Laurent polynomial ring in one variable and a polynomial ring in another two variables  over the $RO(C_2)$-graded Bredon cohomology of a point, generalizing and shedding new light on Voevodsky's computation of the motivic cohomology of $\BG C_2$ (over the reals). Relating these motivic invariants to their equivariant topological counterparts through realization we also obtain previously unknown results in equivariant homotopy, especially about the periodicity of $RO(C_2\times\Sigma_2)$ graded Bredon cohomology of $E_{\Sigma_2}C_2$.
 
 One of the advantages of the computation in the complex case when compared with our paper is that the one dimensional $RO(C_2)$-graded cohomology of a point (computed originally by Stong and used in \cite{HOV2}) is much simpler than the higher dimensional $RO(C_2\times \Sigma_2)$-graded cohomology of point (computed by Holler and Kriz in \cite{HK}  and used in this paper). This makes most of arguments in \cite{HOV2} not to extend to our case.  Another difficulty in the real case as opposed to the complex case is that the $C_2\times \Sigma_2$ topological isotropy sequence that is needed here is more complicated than its complex counterpart and to our knowledge, not previously studied. Both  computations in the cases of real and complex numbers (and therefore, by the usual rigidity theory, of real closed fields and algebraically closed fields of characteristic zero) are part of the understanding of the hard to compute and largely unknown Bredon motivic cohomology ring of an arbitrary field as well as of the $C_2$-equivariant motivic Steenrod algebra of cohomology operations.
  
   Our computations are organized via modules over the $RO(C_2)$-graded Bredon cohomology ring of a point. Before presenting our computations, we recall this ring and introduce some notation used to explain our results. We also recall some basics of equivariant motivic homotopy theory and Bredon motivic cohomology.
   
   \subsection{Equivariant motivic homotopy theory}
   The stable equivariant motivic homotopy category $\SH^{C_2}(k)$ is the stabilization of Voevodsky's category of equivariant motivic spaces  \cite{Del}, with respect to Thom spaces of representations. We recall a few key facts and the notation that we use in the case where $G=C_2$. See \cite{Hoy} or \cite{HOV1} for details.

 Let $V=a+p\sigma$ be  a $C_2$-representation, where $a$ denotes the $a$-dimensional trivial representation and $p\sigma$ is the $p$-dimensional sign representation. 
 We write $\A(V)$ and $\P(V)$ for the  $C_2$-schemes $\A^{\dim(V)}$ and $\P^{\dim(V)-1}$ equipped with the corresponding action coming from $V$. The associated motivic representation sphere is 
 \[
 T^V:=\P(V\oplus 1)/\P(V).
 \] 
 Indexing is based on the following four spheres. There are two topological spheres $S^1$, $S^\sigma$ and two algebro-geometric spheres $S_t=(\A^1\setminus\{0\},1)$ equipped with trivial action, and $S^{\sigma}_t=(\A^1\setminus\{0\},1)$ equipped the 
$C_2$-action $x\rightarrow x^{-1}$.  We write
 \[
 S^{a+p\sigma,b+q\sigma}:=S^{a-b}\wedge S^{(p-q)\sigma}\wedge S^b _t\wedge S^{q\sigma}_t.
 \]
 In this indexing, we have $T\simeq S^{2,1}$ and $T^\sigma\simeq S^{2\sigma, \sigma}$. 
The stable equivariant motivic homotopy category
 $\SH^{C_2}(k)$ is the stabilization of (based) $C_2$-motivic spaces with respect to the motivic sphere 
 $T^{\rho}$ corresponding to the regular representation $\rho= 1 +\sigma$. 
 
 We make use of two fundamental cofiber sequences in 
 $\SH^{C_2}(k)$. The first is
 \begin{equation}\label{eqn:cof1}
 C_{2+}\to S^{0}\to S^{\sigma}.
 \end{equation}
 The second is
 \begin{equation}\label{eqn:cof2}
 \EG C_{2+}\rightarrow S^0\rightarrow \EGt C _2.
 \end{equation}
 Here, $\EG C_{2}$ is the universal free motivic $C_2$-space. It has a geometric model, $\EG C _2= \colim _n\A(n\sigma)\setminus \{0\}$, see \cite[Section 3]{GH}.
The quotient  
$\EG C _2/C _2=\colim _n \left(\A(n\sigma)\setminus \{0\}\right)/C _2$ is the geometric classifying space $\BG C_2$ constructed by Morel-Voevodsky \cite{MVo} and Totaro \cite{T}.  
 Note that  
   $\EGt C _2=\colim _n S^{2n\sigma,n\sigma}$.
 In particular, the maps $S^0\to T^\sigma$ and $S^0\to S^\sigma$ induce equivalences
 \[
 	\EGt C_2 \xto{\simeq} T^{\sigma}\wedge\EGt C_2\,\,\,\textrm{  and  }\,\,\,
 	\EGt C_2 \xto{\simeq} S^{\sigma}\wedge\EGt C_2,
 \]
  see \cite[Proposition 2.9]{HOV1}.

 Equipping a variety with the trivial action yields an embedding $\Sm_{k}\to \Sm^{C_2}_k$ which induces 
a functor $\SH(k)\to \SH^{C_2}(k)$.

   \subsection{Bredon motivic cohomology}
   Bredon motivic cohomology is represented in $\SH^{C_2}(k)$ by  
 the spectrum $M\ul{A}$ associated to an abelian group $A$, where 
 $M\ul{A} _n=A _{tr,C _2}(T^{n\rho})$ is the free presheaf with equivariant transfers, see \cite{HOV1} for details. Here $k$ is an arbitrary field and $\rho$ denotes the $C_2$ regular representation $k[C_2]$.

  \begin{definition}[\cite{HOV1}] The Bredon motivic cohomology of a motivic $C _2$-spectrum $E$ with coefficients in an abelian group $A$ is defined by 
  \[
  \rH^{a+p\sigma,b+q\sigma} _{C_2}(E,A)=[E,S^{a+p\sigma,b+q\sigma}\wedge M\ul{A}] _{\SH^{C _2}(k)}.
  \]
  \end{definition}

If $X\in \Sm^{C_2}_k$ we typically write
\[
H^{a+p\sigma,b+q\sigma}_{C_2}(X,A) := 
\rH^{a+p\sigma,b+q\sigma}_{C_2}(X_+, A).
\]
We call Borel motivic cohomology of a $C_2-$scheme $X$ the Bredon motivic cohomology of the $C_2-$scheme $X\times \EG C_2$.

 When $A$ is a ring, then $H^{\star,\star} _{C _2}(X, A)$ is a graded commutative ring by \cite[Proposition 3.24]{HOV1}. Specifically this means that if $x\in H^{a+p\sigma, b+q\sigma} _{C _2}(X, A)$ and $y\in H^{c+s\sigma, d+t\sigma} _{C _2}(X, A)$, then
\[
x\cup y = (-1)^{ac+ps}y\cup x.
\]
Notice that when $A=\Z/2$, the corresponding Bredon motivic cohomology is a commutative ring.
 
A few features of this theory, which we use are the following (see \cite{HOV1}, \cite{HOV2}).
\begin{itemize}[label = {$\cdot$}]
\item If  $E$ is in the image of $\SH(k)\to \SH^{C_2}(k)$, i.e. it has ``trivial action", then  there is an isomorphism in integral bidegrees with ordinary motivic cohomology,
\[
\rH^{a,b}_{C _2}(E, A) \simeq \rH^{a,b}(E, A).
\]

\item If $X$ has free action, then there is an isomorphism in integral bidegrees with ordinary motivic cohomology,
\[
H^{a,b} _{C _2}(X,A)\simeq H^{a,b}(X/C_2, A).
\]
\item  $H^{\star,\star}_{C_2}(\EG C_2, A)$  is  $(-2+2\sigma,-1+\sigma)$-periodic over any field $k$. The periodicity is given by multiplication with an invertible element $\kappa_2\in H^{2\sigma-2,\sigma-1}_{C_2}(\EG C_2, \ul{A})$ over real numbers. Over complex numbers, we denote the invertible element with $u\in H^{2\sigma-2,\sigma-1}_{C_2}(\EG C_2,A)$.

\item Over complex numbers, Borel motivic cohomology ring is $$H^{\star,\star}_{C_2}(\EG C_2, \Z/2)\simeq \MMt [\tau_\sigma,u^{\pm 1}]$$ with $\tau_\sigma\in H^{0,\sigma}_{C_2}(\EG C_2,\Z/2)$ being the only nontrivial element. We write $\MMt:=H^{*+*\sigma}_{\Br}(\pt, \Z/2)$.

\item The Bredon motivic cohomology groups of complex numbers are given by the following diagram \cite{HOV2}:

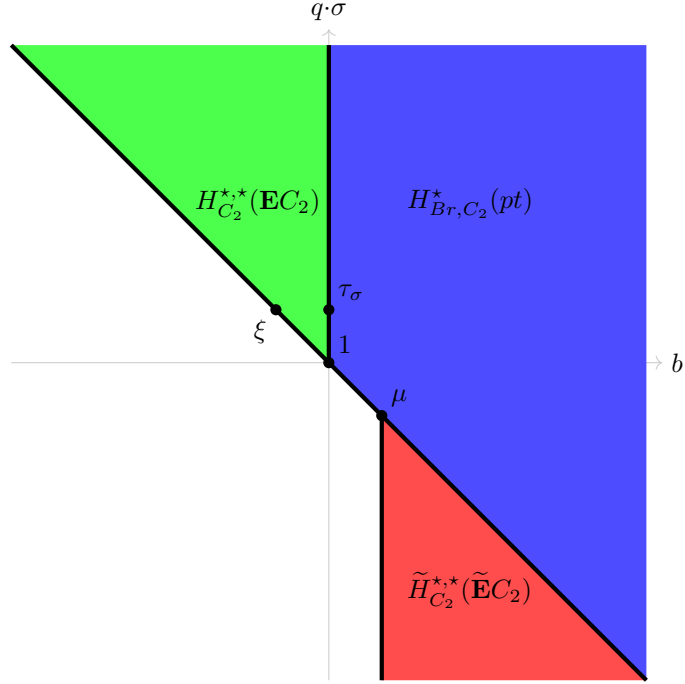
\begin{figure}[H]
	\begin{tikzpicture}[scale=0.7]
		\draw[->, thin, gray!40] (-6.0,0) -- (6.3,0) node[right, black] {$b$};
		\draw[->, thin,  gray!40] (0,-6.0) -- (0,6.3) node[above, black] {$q\cd\sigma$};
		\path[style = myfill1] (0,0) -- (-6,6) -- (0,6) -- (0,0);
		\path[style = myfill2] (0,0) -- (0,6) -- (6,6) -- (6,-6) -- (0,0);
		\path[style = myfill3] (1,-1) -- (1,-6) -- (6,-6) -- (1,-1);
		\draw[ultra thick] (0,0) -- (-6,6);
		\draw[ultra thick] (0,0) -- (0,6);
		\draw[ultra thick] (0,0) -- (6,-6);
		\draw[ultra thick] (1,-1) -- (1,-6); 
		\node at (-2.7,3) [right] {$H^{\star,\star}_{C_2}(\EG C_2)$};
		\node at (1.3,3) [right] {$H^{\star}_{\Br,C_2}(\pt)$};
		\node at (1.3,-4.3) [right] {$\rH^{\star,\star}_{C_2}(\wt\EG C_2)$};
		\fill (0,1) circle(3pt) node[above right] {$\tau_\sigma$};
		\fill (0,0) circle(3pt) node[above right] {$1$};
		\fill (-1,1) circle(3pt) node[below left] {$\xi$};
		\fill (1,-1) circle(3pt) node[above right] {$\mu$};
		
	\end{tikzpicture} \hfill 
	\caption{Regions of {$H^{\star, b+q\sigma}_{C_2}(\C,\Z/2)$} determined by $\EG C_2$, Betti realization, and $\wt\EG C_2$. The degrees of the displayed elements are $|\xi|=(-2+2\sigma, -1+\sigma)$, $|\mu| = (0,1-\sigma)$, $|\tau_\sigma| = (0,\sigma)$.   }
	\label{TEST0}
\end{figure}
\end{itemize}
A point $(b,q)$ in the Figure \ref{TEST0} is given by the graded $\Z/2$-vector space $\oplus_{a,p\in \Z^2}H^{a+p\sigma,b+q\sigma}_{C_2}(\C,\Z/2)$. At each point $(b,q)$ we have a realization map $$ H^{\st,b+q\sigma}_{C_2}(\C,\Z/2)\rightarrow H^{\st}_{Br}(pt,\Z/2).$$

The map of sites 
  $\Sm _\R^{C_2}\rightarrow {\rm Top}^{C_2\times \Sigma_2}$, given by $X\rightarrow X(\C)$, where the set of complex points is equipped with the analytic topology,  extends to a functor
  $\Re:\SH^{C _2}(\R)\rightarrow \SH^{C _2\times \Sigma_2}$ between the stable equivariant motivic homotopy category over $\R$ and the classical stable equivariant homotopy category. We refer to this functor as ``Betti realization," or simply ``realization".
  
  The way the motivic spheres interact with the topological spheres through Betti realization is
  
$$\Re (S^{a+p\sigma,b+q\sigma})=S^{a-b+(p-q)\sigma+b\epsilon+q\sigma\otimes\epsilon}.$$  
Here we have the following nontrivial one dimensional $C_2\times \Sigma_2$-representations: $\sigma$ ($C_2$ nontrivial action for the first component), $\epsilon$ ($C_2$ nontrivial action for the second component) and $\sigma\otimes \epsilon$. We have the following four maps giving four $C_2\times \Sigma_2$-one dimensional irreducible representations
$$\Z/2\times\Z/2\rightarrow \Z/2\hookrightarrow GL_1(\R)$$
given by $(k_1,k_2)\rightarrow ik_1+jk_2$ where $i,j\in\{0,1\}$ are fixed choices. Now $\sigma$ is the choice $(i,j)=(1,0)$, $\epsilon$ is the choice $(i,j)=(0,1)$ and $\sigma\otimes \epsilon$ is the choice $(i,j)=(1,1)$. The identity representation is given by the choice $(i,j)=(0,0)$. The way $\Z/2=\{0,1\}$ embeds in $GL_1(\R)$
 is by sending $0$ to multiplication by $1$ and sending $1$ to multiplication by $-1$.
 
  By \cite[Theorem A.29]{HOV1}, $\Re(M\ul{A})  \simeq H\ul{A}$, where $H\ul{A}$ is the equivariant Eilenberg-MacLane spectrum associated to the constant Mackey functor $\ul{A}$. 
  
   In particular, for any smooth $C _2$-scheme over $\R$ there is a realization map
\[
\Re :H^{a+p\sigma,b+q\sigma} _{C _2}(X, A)\rightarrow H^{a-b+(p-q)\sigma+b\epsilon+q\sigma\otimes\epsilon}_{\Br}(X(\C),A).
\] 
  We write $S(V)$ for the unit sphere included in the disk $D(V)$ given by any actual $C_2\times\Sigma_2$-representation $V$. Let $E_{\Sigma_2}C_2$ is the $\Sigma_2$-equivariant universal free $C_2$-space (see [\cite{May},VII.1]) given by $$E_{\Sigma_2}C_2=colim_n S(n\sigma+n\sigma\otimes\epsilon).$$
  
Betti realization takes the cofiber sequences \eqref{eqn:cof1} and \eqref{eqn:cof2} to the corresponding ones in $\SH^{C_2\times \Sigma_2}$.  These are 
 \begin{equation}\label{eqn:cof3}
 C_{2+}\to S^{0}\to S^{\sigma}.
 \end{equation}
and 
 \begin{equation}\label{eqn:cof4}
  E_{\Sigma_2}C_{2+}\rightarrow S^0\rightarrow \tilde{E}_{\Sigma_2} C _2.
 \end{equation}
The cofiber sequence \ref{eqn:cof3} is associated to the $C_2\times \Sigma_2$ representation $\sigma$; consequently we have by symmetry two more similar topological cofiber sequences associated to the $C_2\times \Sigma_2$ representations $\epsilon$ and $\sigma\otimes \epsilon$.  The same happens for the cofiber sequence \ref{eqn:cof4}; there are two other symmetric topological cofiber sequences depending on which two one dimensional $C_2\times \Sigma_2$ representations are chosen. According to $\cite{HOV1}$, over the reals we have 
 $$Re(\EG C_2)=E_{\Sigma_2}C_2,$$ where by $\Sigma_2$ we denote the second copy of $C_2$ in $C_2\times C_2$.

This implies that the Bredon cohomology  of the $C_2\times\Sigma_2$-space $E_{\Sigma_2}C_2$ has a $-1+\sigma-\epsilon+\sigma\otimes\epsilon$ periodicity because the Bredon motivic cohomology of $\EG C_2$ is $(2\sigma-2,\sigma-1)$ periodic [\cite{HOV1}, Theorem 5.4]. According to the topological realization, we also have that $\tilde{E}_{\Sigma_2}C_2$ has, in its reduced Bredon cohomology, periodicities $\sigma$ and $\sigma\otimes\epsilon$ because the reduced Bredon motivic cohomology of $\wt\EG C_2$ is $(0,\sigma)$ and $(\sigma,0)$ periodic [\cite{HOV1}, Proposition 5.7].
This is because $\tilde{E}_{\Sigma_2}C_2$ is the Betti realization of $\wt\EG C_2$. 

Throughout this paper, we refer to the cofiber sequence \ref{eqn:cof4} as the ``$C_2\times \Sigma_2$ topological isotropy sequence".

\subsection{$RO(C_2)$-graded Bredon cohomology of a point}
 We present our computations as modules over the $RO(C_2)$-graded Bredon cohomology of a point. The natural module structure is given by the fact that Betti realization induces an isomorphism of bigraded rings 
$$H^{\star,0}_{C _2}(\R,\Z/n)\simeq \MMtn,$$ and 
so $H^{\star,\star}_{C_2}(X,\Z/n)$ is a module over $\MMtn$. 

In fact, by \cite{V},  Betti realization is an isomorphism in weight zero, even with $\Z$-coefficients. In Section 4, we will study a more general case ($q\in \Z$) than the proposition below ($q=0$). 
The following proposition is a well known Voevodsky's computation. We use a notation in which $RO(C_2)-$Bredon cohomology ring of a point is viewed as a subring in $RO(C_2\times \Sigma_2)-$Bredon cohomology ring of a point induced by the $C_2\times \Sigma_2$ nontrivial irreducible representation $\epsilon$. 

\begin{proposition}\label{classic} $H^{a,b}(\R,\Z/2)\simeq H^{a-b+b\epsilon} _{Br}(pt,\Z/2)$ for any $a\in \Z, b\geq 0$. Moreover, as rings, $$H^{*,*}(\R,\Z/2)\simeq \Z/2[x_2,y_2]$$ where the last expression is the positive cone  of the Bredon cohomology of a point (i.e. $a+p\geq 0$, $p\geq 0$ in Figure \ref{TEST}) with $deg(x_2)=(1,1)$ and $deg(y_2)=(0,1)$. The realization maps are monomorphisms, but not isomorphisms if $b+1<a\leq 0$.
\end{proposition}

\begin{proposition}\label{prop:ptiso}
Let $A$ be a finite abelian group and $b\geq 0$. For any $a,p\in \Z$, Betti realization induces an isomorphism
\[
H^{a+p\sigma, b}_{C_2}(\R, A)\xrightarrow{\simeq} H^{a-b+p\sigma+b\epsilon}_{\Br}(\pt,A).
\]
\end{proposition}
\begin{proof}
If $b\geq 0$, then $H^{a,b}(\R, A)\to H^{a-b+b\epsilon}_{Br}(pt, A)$ is an isomorphism for all $a\in \Z$ according to Voevodsky's computation in Proposition \ref{classic}. In particular, the result holds for $p=0$.  Using the comparison long exact sequence produced from the realization of the cofiber sequence \ref{eqn:cof1} and the five lemma, the result holds for all $p$ by induction. For example we have the diagram
\[\begin{tikzcd}
	{H^{a,b}_{C_2}(C_2)} & {H^{a,b}_{C_2}(\R)} & {H^{a+\sigma,b}_{C_2}(\R)} & {H^{a+1,b}_{C_2}(C_2)} & {H^{a+1,b}_{C_2}(\R)} \\
	{H^{a-b+b\epsilon}_{Br,C_2}(C_2)} & {H^{a-b+b\epsilon}_{Br,C_2}(pt)} & {H^{a-b+\sigma+b\epsilon}_{Br,K}(pt)} & {H^{a+1-b+b\epsilon}_{Br,C_2}(C_2)} & {H^{a+1-b\epsilon}_{Br,C_2}(pt)}
	\arrow[from=1-1, to=1-2]
	\arrow["\simeq"', from=1-1, to=2-1]
	\arrow[from=1-2, to=1-3]
	\arrow["\simeq"', from=1-2, to=2-2]
	\arrow[from=1-3, to=1-4]
	\arrow[from=1-3, to=2-3]
	\arrow[from=1-4, to=1-5]
	\arrow["\simeq"', from=1-4, to=2-4]
	\arrow["\simeq"', from=1-5, to=2-5]
	\arrow[from=2-1, to=2-2]
	\arrow[from=2-2, to=2-3]
	\arrow[from=2-3, to=2-4]
	\arrow[from=2-4, to=2-5]
\end{tikzcd}\]
that induces isomorphism in the bidegree $(a+\sigma,b)$. Inductively, we obtain the isomorphism on all bidegrees $(a+p\sigma,b)$, $p\geq 0$. For the bidegrees $(a-p\sigma,b)$, $p\geq 0$ we look to the following diagram
\begin{tikzcd}
	{H^{a-1,b}(pt)} & {H^{a-1,b}(C_2)} & {H^{a-\sigma,b}(pt)} & {H^{a,b}(pt)} & {H^{a,b}(C_2)} \\
	{H^{a-1-b+b\epsilon}_{Br,C_2}(pt)} & {H^{a-1-b+b\epsilon}_{Br,C_2}(pt)} & {H^{a-b-\sigma+b\epsilon}_{Br,K}(pt)} & {H^{a-b+b\epsilon}_{Br,C_2}(pt)} & {H^{a-b+b\epsilon}_{Br,C_2}(pt)}
	\arrow[from=1-1, to=1-2]
	\arrow["\simeq", from=1-1, to=2-1]
	\arrow[from=1-2, to=1-3]
	\arrow["\simeq", from=1-2, to=2-2]
	\arrow[from=1-3, to=1-4]
	\arrow[from=1-3, to=2-3]
	\arrow[from=1-4, to=1-5]
	\arrow["\simeq",from=1-4, to=2-4]
	\arrow["\simeq",from=1-5, to=2-5]
	\arrow[from=2-1, to=2-2]
	\arrow[from=2-2, to=2-3]
	\arrow[from=2-3, to=2-4]
	\arrow[from=2-4, to=2-5]
\end{tikzcd}
and inductively we obtain the isomorphism on all bidegrees $(a-p\sigma,b)$, $p\geq 0$.
\end{proof}

   The $RO(C_2)$-graded Bredon cohomology groups of a point are shown in the diagram below. We notice that both cones of the $RO(C_2)$-Bredon cohomology of a point in Figure \ref{TEST} are parts of the $RO(C_2)$-graded Bredon cohomology of $EC_2$ and $\tilde{E}C_2$ (see for example \cite{HOV2}).
\begin{figure}[H]
	\begin{tikzpicture}[scale=0.7]
		\draw[->, thin, gray!40] (-6.0,0) -- (6.3,0) node[right, black] {$a$};
		\draw[->, thin,  gray!40] (0,-6.0) -- (0,6.3) node[above, black] {$p\cd\sigma$};
		\path[style = myfill1] (0,0) -- (-6,6) -- (0,6) -- (0,0);
		\path[style = myfill3] (1,-1) -- (1,-6) -- (6,-6) -- (1,-1);
		\draw[ultra thick] (0,0) -- (-6,6);
		\draw[ultra thick] (0,0) -- (0,6);
		\draw[ultra thick] (0,0) -- (6,-6);
		\draw[ultra thick] (1,-1) -- (1,-6); 
		\node at (-3.5,4) [right] {$H^{a+p\sigma}_{Br}(EC_2)$};
		\node at (1.3,-4.8) [right] {$\rH^{a+p\sigma} _{Br}(\tilde{E}C_2)$};
		\fill (0,0) circle(3pt) node[above right] {$1$};
		\fill (-1,1) circle(3pt) node[below left] {$\alpha$};
		\fill (1,-1) circle(3pt) node[above right] {$\theta$};
		
	\end{tikzpicture} \hfill 
	\caption{Regions of {$H^{*+*\sigma}_{Br}(\R,\Z/2)$} determined by (parts) of the $RO(C_2)-$ graded Bredon cohomology of  $EC_2$ and $\tilde{E}C_2$. The degrees of the displayed elements are $|\alpha|=-1+\sigma$, $|\theta|=2-2\sigma$. }  
	\label{TEST}
	\end{figure}
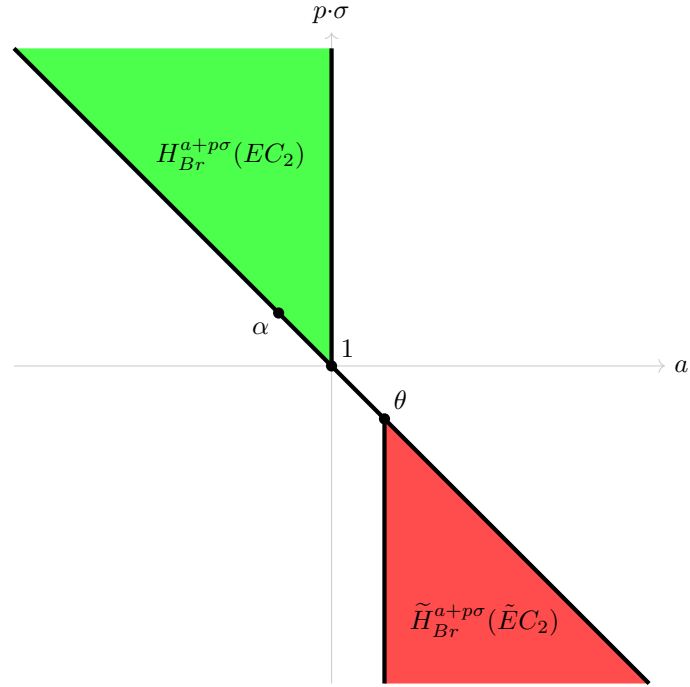

A point $(a,p)$ in Figure \ref{TEST} is given by a single $\Z/2$ vector space $H^{a+p\sigma}_{Br}(pt,\Z/2)$. The diagonal in Figure \ref{TEST} is the line $a+p=0$. The positive cone $R$ is the region $b+q\geq 0,b<0,q\geq 0$  and it is a subring that can be computed as $\Z/2[\sigma,\alpha]$ with $|\sigma|=\sigma$ and $|\alpha|=-1+\sigma$.
The negative cone $NC$ ($a+p\leq 0$, $a\geq 1$ in Figure \ref{TEST}) is computed as $\Z/2\{\frac{\theta}{\sigma^n\alpha^m}\}$ for $n,m\geq 0$ and $\theta$ an element divisible by $\alpha$ and $\sigma$ with the property that $\theta^2=\theta\alpha=\theta\sigma=0$ and $|\theta|=2-2\sigma$. Also, the cohomology class $\alpha$ of the $RO(C_2)$-graded Bredon cohomology ring of $EC_2$ is invertible.

A detailed explanation of the above diagram is given in \cite{HOV2}.

The $RO(C_2)$-graded Bredon cohomology ring of a point was computed by Stong (unpublished), but written accounts can be found in the Appendix of \cite{Car:op} and in Proposition 6.2 of \cite{dug:kr}. This can be described as:
\begin{theorem} \label{Br} We  have a $\MMt$-algebra isomorphism $$H^{*+*\sigma}_{Br}(pt,\Z/2)=R\oplus NC\simeq \Z/2[\sigma,\alpha]\oplus \Z/2\left\{\frac{\theta}{\sigma^n\alpha^m}\right\}$$
with $R$ a subring of $H^{*+*\sigma}_{Br}(pt,\Z/2)$ and $NC$ a $\MMt$-submodule with zero products. 
\end{theorem}
\subsection{Our Results}

 One of the main results of this paper is the additive picture of the Bredon motivic cohomology groups of the real numbers. As one can notice in Figure \ref{TEST1}, apart from different Betti realizations in the region $b,q\geq 0$, $b+q\geq 0$ (and different computations for the graded $\Z/2$-vector spaces in the regions $b+q\geq 0$, $b<0$ and $b+q<0$, $b\geq 1$), the diagram is similar with the complex case computed in \cite{HOV2} and reviewed in Figure \ref{TEST0}. This shows that in both the complex and real cases the cohomology groups are only determined by parts of Bredon motivic cohomologies of $\EG C_2$ and $\wt \EG C_2$, along with the realization.
 
 The classical computations of the motivic cohomology of the complex numbers and real numbers (with $\Z/2$-coefficients) can be understood as belonging to the region $b,q\geq 0$, $b+q\geq 0$ (more precisely to the line $q=0$, $b\geq 0$), an area where all the Betti realizations are isomorphisms and to the segment $b<0$, $q=0$ where the Betti realizations can be monomorphisms that are not isomorphisms or isomorphisms between two zero groups (trivial isomorphisms). 
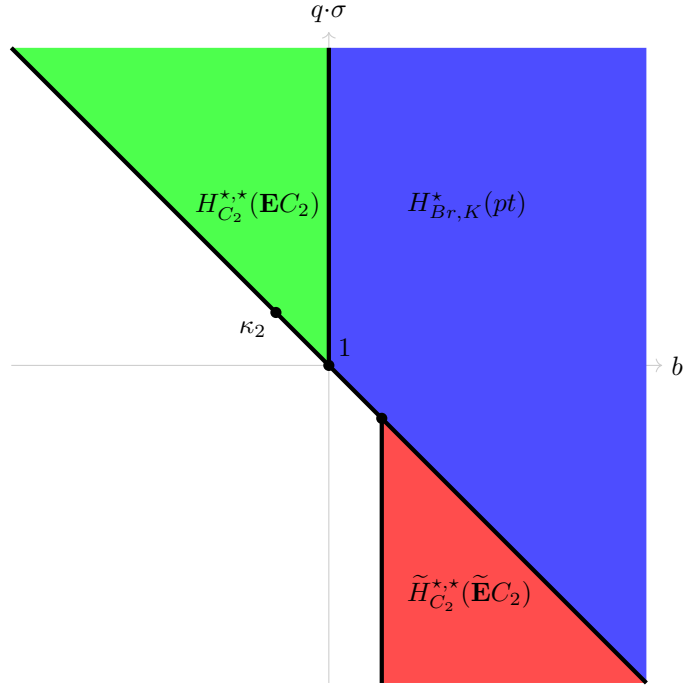
\begin{figure}[H]
	\begin{tikzpicture}[scale=0.7]
		\draw[->, thin, gray!40] (-6.0,0) -- (6.3,0) node[right, black] {$b$};
		\draw[->, thin,  gray!40] (0,-6.0) -- (0,6.3) node[above, black] {$q\cd\sigma$};
		\path[style = myfill1] (0,0) -- (-6,6) -- (0,6) -- (0,0);
		\path[style = myfill2] (0,0) -- (0,6) -- (6,6) -- (6,-6) -- (0,0);
		\path[style = myfill3] (1,-1) -- (1,-6) -- (6,-6) -- (1,-1);
		\draw[ultra thick] (0,0) -- (-6,6);
		\draw[ultra thick] (0,0) -- (0,6);
		\draw[ultra thick] (0,0) -- (6,-6);
		\draw[ultra thick] (1,-1) -- (1,-6); 
		\node at (-2.7,3) [right] {$H^{\star,\star}_{C_2}(\EG C_2)$};
		\node at (1.3,3) [right] {$H^{\st}_{\Br,K}(\pt)$};
		\node at (1.3,-4.3) [right] {$\rH^{\star,\star}_{C_2}(\wt\EG C_2)$};
		\fill (0,0) circle(3pt) node[above right] {$1$};
		\fill (-1,1) circle(3pt) node[below left] {$\kappa_2$};
		\fill (1,-1) circle(3pt) node[above right] {};
		
	\end{tikzpicture} \hfill 
	\caption{Regions of {$H^{\star, b+q\sigma}_{C _2}(\R,\Z/2)$} determined by $\EG C_2$, Betti realization into the $RO(C_2\times \Sigma_2)$-graded Bredon cohomology of a point, and $\wt\EG C_2$. The degree of the displayed element is $|\kappa_2|=(-2+2\sigma, -1+\sigma)$. }  
	\label{TEST1}
	\end{figure}
	
	A point $(b,q)$ in the Figure \ref{TEST1} is given by the graded $\Z/2-$vector space $\oplus_{a,p\in \Z^2}H^{a+p\sigma,b+q\sigma}_{C_2}(\R,\Z/2)$. We can also view each point of the diagram as a $\MMt$-module.
	
	In the region $b+q<0$, $b\leq 0$ of Figure \ref{TEST1}, the Bredon motivic cohomology groups are zero. Multiplication by $\kappa_2$ gives the periodicity in the Bredon motivic cohomology of $\EG C _2$; therefore, the multiplication by $\kappa_2$ is an isomorphism in Figure \ref{TEST1} only in the region $b<0$, $b+q\geq 0$. The Bredon motivic cohomology of the real numbers is isomorphic with the Bredon motivic cohomology of $\EG C_2$ in the region $b<0$, $b+q\geq 0$, and it is isomorphic to the reduced Bredon motivic cohomology of $\wt\EG C_2$ in the region $b+q<0$, $b\geq 1$.  The region $b,q\geq 0$, $b+q\geq 0$ includes both the computation of the motivic cohomology of $\R$ of V.Voevodsky \cite{Voev:miln} and the results in codimension 0, 1 and $\sigma$ in the particular case $\R$ (or any real closed field), of the second author in \cite{V}.

Moreover, in Figure \ref{TEST1}, we prove that the realization maps in the region $q\geq 0$, $b<0$,  $b+q\geq 0$ and the region $b+q<0$, $b\geq 1$ can be identified with maps induced by the $C_2\times \Sigma_2$ topological isotropy sequence. By studying the $C_2\times \Sigma_2$ topological isotropy sequence in sections 2.3 and 3, we completely determine the realization maps in the above regions.

 Based on J.Holler and I.Kriz's computation from \cite{HK}, we also identify the value of the cohomology groups in Figure \ref{TEST1} above the line $b+q=0$. The value of the cohomology below the line $b+q=0$ can be determined from Voevodsky's computation of the motivic cohomology of $\BG C_2$ over reals (\cite{Voc}). Using Theorem \ref{ocomp} below (proved in Section 5), we can also compute these groups as $RO(C_2\times\Sigma_2)$-graded Bredon cohomology groups therefore reducing the computation to \cite{HK}. This unexpected link between Voevodsky's computation and $RO(G)$-graded Bredon cohomology is discussed in Remark \ref{repr} from Section 5 (for $G=C_2$ or $C_2\times\Sigma_2$); it is essentially a consequence of the fact that there is a lot of nontrivial information in the extra nontrivial representation indexes of Bredon motivic cohomology.

We prove the following theorem about the realization maps in Figure \ref{TEST1}:
\begin{theorem} The Betti realizations 
$$H^{a+p\sigma,b+q\sigma}_{C_2}(\R,\Z/2)\rightarrow H^{a-b+(p-q)\sigma+b\epsilon+q\sigma\otimes\epsilon}_{Br}(pt,\Z/2)$$
in the region $b+q\geq 0$, $q\geq 0$, $b<0$ are monomorphisms everywhere and isomorphisms if $a\leq 2b+2$ and the Betti realizations in the region $b,q\geq 0$, $b+q\geq 0$ are isomorphisms everywhere. The realization maps in the region $b+q<0$, $b\geq 1$ are monomorphisms everywhere.
\end{theorem}
Thus, we completely compute the Borel motivic cohomology ring of the real numbers as a proper subring of the $RO(C_2\times \Sigma_2)$-graded Bredon cohomology of $E_{\Sigma_2}{C_2}$ (which we compute in \cite{DV} and mention in section 3). The method for computing Borel motivic cohomology of the complex numbers in \cite{HOV2} cannot be generalized to our case, but we found a different method that computes Borel cohomology of both the real numbers and the complex numbers, therefore also reproving in a simpler and more conceptual way the results of \cite{HOV2}. One of the main difficulties in the case of the real numbers as opposed to the case of the complex numbers is that the $\Z/2$-vector space dimensions of the $RO(C_2\times \Sigma_2)$-graded Bredon cohomology groups of a point (with $\Z/2$-coefficients) are usually higher than 1. 

We prove the following theorem in Section 5:
\begin{theorem} \label{ocomp} We have the following ring isomorphism  $$H^{\st,\st}_{C_2}(\EG C _2,\Z/2)\simeq \MMt[x _3,y _3,\kappa_2^{\pm 1}]\hookrightarrow H^{*+*\sigma+*\epsilon+*\sigma\otimes\epsilon}_{Br}(E_{\Sigma_2}C_2,\Z/2),$$ 
with $x_3$ in bidegree $(\sigma,\sigma)$, $y_3$ in bidegree $(\sigma-1,\sigma)$ and $\kappa_2$, the invertible element, in bidegree $(2\sigma-2,\sigma-1)$.
\end{theorem} 
The realization maps $$H^{a+p\sigma,b+q\sigma}_{C_2}(\EG C _2,\Z/2)\rightarrow H^{a-b+(p-q)\sigma+b\epsilon+q\sigma\otimes\epsilon} _{Br}(E_{\Sigma_2}C_2,\Z/2)$$ are isomorphisms if $a\leq 2b+2$, and injectives, but not surjectives if $a\geq 2b+3$ (see Corollary \ref{EG}). We explain in Remark \ref{Vcomp} and Remark \ref{repr} how V.Voevodsky's computation of the motivic cohomology of $\BG C_2$ \cite{Voc} can be embedded, for our choices of a field, in the Borel motivic cohomology ring of a point.

 We notice that the multiplication by the generator $\tau_{\sigma}\in H^{0,\sigma}_{C_2}({\EG C_2} _{\C},\Z/2)$ gives isomorphism over the complex numbers, and is therefore important in the computation of Borel motivic cohomology ring of complex numbers in \cite{HOV2}. However, in the real case this cohomology class is zero because $H^{0,\sigma}_{C_2}({\EG C_2} _{\R},\Z/2)=0$.

Using the above results we give a description of the Bredon motivic cohomology ring of the real numbers as a subring in the $RO(C_2\times \Sigma_2)$-graded Bredon cohomology ring of a point. We consider $$R=\oplus_{b\geq 0,b+q\geq 0}H^{\st,b+q\sigma} _{C_2}(\R,\Z/2)$$ the subring in $H^{\st,\st}_{C_2}(\R,\Z/2)$ given by the direct sum of the $\MMt$-modules in the region $b,q\geq 0$, $b+q\geq 0$ of Figure \ref{TEST1} (which can be considered as a subring of the $RO(C_2\times \Sigma_2)$-graded Bredon cohomology of a point, and is discussed in section 6) and $$NC=\oplus _{b\geq 0, b+q<0}H^{\st,b+q\sigma} _{C_2}(\R,\Z/2),$$ the $\MMt$-submodule in $H^{\st,\st}_{C_2}(\R,\Z/2)$  given by the  direct sum of the $\MMt$-modules in the region $b+q<0$, $b\geq 1$ of Figure \ref{TEST1}.  I recall that we have an identification $\MMt=H^*_{Br,C_2}(pt,\Z/2)\simeq H^{*,0}_{C_2}(\R,\Z/2)$ and the $\MMt$-module structures of $R$ (weight $b+q\geq 0$, $b\geq 0$) and $NC$ (weight $b+q<0$, $b\geq 1$) are given by the fact that multiplication by $\MMt$ (via the identification) doesn't change the weight; it changes only the cohomological degree.

In terms of generators (see section 6) we write
$$NC\simeq \oplus _{a\leq 2b+1} x^{-b-1}_3\rH^{a+b\epsilon}_{Br,C_2}(\Sigma B_{\Sigma_2}C_2,\Z/2)[x^{\pm 1}_1,x^{-1}_3]=$$
 $$=\{x^{-n-m-p-2}_3x^n_2y^m_2\Sigma (b^pc),x^{-n-m-p-1}_3x^n_2y^m_2\Sigma (b^{p+1})\}[x^{\pm 1}_1,x^{-1}_3].$$
Here $x_3$ is in bidegree $(\sigma,\sigma)$, $x_2$ is in bidegree $(1,1)$, $y_2$ is in bidegree $(0,1)$,  $\kappa_2$, the invertible element, in bidegree $(2\sigma-2,\sigma-1)$ and $\Sigma b^{p+1}$ is in bidegree $(p+2)+(p+1)\epsilon$ and  $\Sigma b^pc$ is in bidegree $(p+1)+(p+1)\epsilon$.
Because of the description of $NC$ in terms of generators we can identify its $\MMt$-module structure in terms of these generators. This is given by $$\frac{\theta_1}{x_1^{n_1}y_1^{m_1}}\alpha=0;$$ and $$y_1\frac{x^n_2y^m_2\Sigma (b^pc)}{x^{n+m+p+2}_3}=x_1\frac{x^{n+1}_2y_2^m\Sigma cb^{p-1}}{x_3^{n+m+p+2}}+x_1\frac{y_2^{m+1}x_2^n\Sigma b^p}{x_3^{n+m+p+2}},p\geq 1;$$ and 
$$y_1\frac{x^n_2y^m_2\Sigma b^p}{x^{n+m+p+1}_3}=x_1\frac{x_2^{n}y_2^{m}\Sigma cb^{p-1}}{x^{n+m+p+1}_3},p\geq 1;$$ and
$$y_1\frac{x^n_2y^m_2\Sigma c}{x^{n+m+1}_3}=0.$$

The Bredon motivic cohomology of the real numbers is described below in terms of these objects:

\begin{theorem} We have a decomposition of  $\MMt$-modules
$$H^{\st,\st} _{C_2}(\R,\Z/2)=(R,\kappa_2)\oplus NC$$
with $\kappa_2$ in degree $(2\sigma-2,\sigma-1)$ with the realization map making $(R,\kappa_2)\hookrightarrow H^\st_{Br,K}(pt,\Z/2)$ a subring of the $RO(C_2\times \Sigma_2)$-graded Bredon cohomology of a point, generated by $R$ and the cohomology class $\kappa_2$ and with $NC$ a $\MMt-$submodule of $H^\st_{Br,K}(pt,\Z/2)$  having zero products. 

The realization map is a monomorphism making the Bredon motivic cohomology ring of the real numbers a nontrivial proper subring of the $RO(C_2\times \Sigma_2)$-graded Bredon cohomology ring of a point. 

 In conclusion we have the following commutative diagram of commutative rings
\[\begin{tikzcd}
	{H^{*,*}(\R,\Z/2)} & {H^{\st} _{Br,C_2}(pt,\Z/2)} \\
	{H^{\star,\star} _{C_2}(\R,\Z/2)} & {H^{\st} _{Br,K}(pt,\Z/2)}
	\arrow[tail, from=1-1, to=1-2]
	\arrow[tail, from=1-2, to=2-2]
	\arrow[tail, from=1-1, to=2-1]
	\arrow[tail, from=2-1, to=2-2]
\end{tikzcd}\]
The horizontal maps are given by the Betti realizations and the vertical maps by the canonical inclusions.
\end{theorem} 
The motivic cohomology ring of the real numbers was computed by V.Voevodsky as the positive cone of the $RO(C_2)$-graded Bredon cohomology of a point and it is a subring of $R$. Voevodsky's monomorphism is represented by the upper horizontal Betti realization in the above diagram.  

The description of the subring $H^{\st,\st} _{C_2}(\R,\Z/2)$ in terms of generators and relations depends on the description of the ring $H^\st_{Br,K}(pt,\Z/2)$ in terms of the generators and relations. This was started in \cite{BH} and finalized in \cite{DV}. In \cite{DV} there is a more detailed description of the subring $H^{\st,\st} _{C_2}(\R,\Z/2)$ in terms of generators and relations.

We can conclude that Bredon and Borel motivic cohomology groups of real numbers are completely determined by the values of $RO(C_2\times\Sigma_2)$ graded Bredon cohomology groups of a point. In the same way, we can conclude that Bredon and Borel motivic cohomology groups of complex numbers  are completely determined by the values of $RO(C_2)$ graded Bredon cohomology groups of a point. 

At the end of the paper, we show that all the computations and results of this paper are valid if we replace $\R$ by an arbitrary real closed field.

A brief outline of the paper is as follows. Sections 1 and 2 are devoted to the introduction and preliminaries. The main computations of the Bredon motivic cohomology of real numbers are carried out in Sections 3, 4, 5 and 6.
We study the realization maps of  Figure \ref{TEST1} by dividing the discussion on three ranges of weight:
\begin{itemize}
\item $b+q<0$,
\item $b\geq 0$, $b+q\geq 0$,
\item $b<0$, $b+q\geq 0$
\end{itemize}
We prove in section 3 and subsection 4.1. that the realization maps in the first range above are monomorphisms, but not necessary isomorphisms. We prove in section 4 that the realization maps in the second range above are isomorphisms. In section 5 we discuss the Bredon motivic cohomology of $\EG C_2$ over the real numbers, which is usually called the Borel motivic cohomology of real numbers and as an application we conclude that the realization maps in the third range above are monomorphisms, but not necessary isomorphisms. In section 6, based on the previous sections, we prove and conclude properties of the Bredon motivic cohomology  ring of real numbers. In subsection 6.1. we discuss the general case of a real closed field.\\
  
  {\bf{Notation.}} 
  \begin{itemize}[label = {$\cdot$}]
 \item We write $X_k$ for a smooth scheme over $k$.
 \item $K:=C _2\times \Sigma _2$ is the Klein four-group. We denote by $\Sigma_2$ the second copy of $C_2$.
\item $H^{a+p\sigma,b+q\sigma} _{C _2}(X,A)$ is the Bredon motivic cohomology of a $C _2$-smooth scheme $X$, with coefficients in an abelian group $A$.  All cohomologies that appear in this paper are understood to be with $\Z/2$ coefficients.
\item $H^{n,q}(X,A)$ is the motivic cohomology of a smooth scheme $X$ with coefficients in an abelian group $A$. We only consider the case where $A=\Z/2$.
\item $H^{a+p\sigma}_{\Br,C_2}(X,A)$ is the Bredon cohomology of a $C_2$-topological space $X$ with coefficients in the constant Mackey functor $\underline{A}$ (and generated by the classes $x_1,y_1,\theta_1$). If instead of $\sigma$ we write $\epsilon,$ we mean the same cohomology generated by $x_2,y_2, \theta_2$ (viewed as being embedded in the $RO(C_2\times \Sigma_2)$-graded Bredon cohomology of a point). We only consider the case where $A=\Z/2$. We write $H^{a+p\sigma}_{\Br,C_2}(X,A)$ for the Bredon cohomology associated to the $C_2$-(virtual) representation $a+p\sigma$ rather than $a-p+p\sigma$ (several papers in the equivariant homotopy field use this second notation). Our chosen notation is consistent with the notation of $\cite{HOV1}$, $\cite{HOV2}$ and $\cite{Car:op}$.
\item $H^{a+p\sigma+b\epsilon+q\sigma\otimes\epsilon}_{\Br,K}(X,A)$ is the Bredon cohomology of a $C_2\times \Sigma_2$-topological space $X$ with coefficients in the constant Mackey functor $\underline{A}$. We only consider the case where $\underline{A}=\underline{\Z/2}$. Sometimes we write $H^{\st}_{\Br,K}(X,A)$ for the same cohomology, where $K$ is the Klein four-group.
\item We have the convention that $\st$ denotes a $RO(G)$-grading, while $*$ denotes an integer grading. 

For example, $H^{\star,\star} _{C _2}(X)=\oplus_{a,b,p,q}  H^{a+p\sigma,b+q\sigma} _{C _2}(X)$ and  $H^{*,*}(X)=\oplus_{a,b} H^{a,b}(X)$.  
\item We write $\ul{H}^{Br}_{V}$, $\ul{H}^V_{Br}$ for the usual Mackey functors associated to $H^{Br}_{V}$ or $H^V_{Br}$, with $V$ a $G-$representation.
  \item  $S^V$ is the topological sphere associated to the $C_2\times \Sigma_2$ representation $V$.  For example, $V$ can be $\sigma$, $\epsilon$ or $\sigma\otimes\epsilon$.
\item  All $C_2$-varieties are over $\R$, and we view $C_2$ as the group scheme $C_2=\spec(\R)\sqcup \spec(\R)$.
 \item $\MMtn:=H^{*+*\sigma}_{\Br}(\pt, \Z/n)$. 
 \item We denote by $\kappa_2$ the invertible element in the Bredon motivic cohomology of $\EG C_2$  (of degree $(2\sigma-2,\sigma-1)$) as well as its Betti realization in the $RO(C_2\times \Sigma_2)$-graded Bredon cohomology of $E_{\Sigma_2}C_2$, where the degree is $-1+\sigma-\epsilon+\sigma\otimes\epsilon$ (this cohomology class gives the $C_2\times \Sigma_2$ Bredon cohomology of $E_{\Sigma_2}C_2$ its unique periodicity). 
 \newpage
 \item If the $b^nc$ and $b^n$ are cohomology classes in $\rH^{*+*\epsilon}_{Br,C_2}(B_{\Sigma_2}C_2)$, then we denote by $\Sigma b^nc$ and $\Sigma b^n$ their corresponding cohomology classes in the cohomology of the suspension $\rH^{*+1+*\epsilon}_{Br,C_2}(\Sigma^1B_{\Sigma_2}C_{2}).$
  \end{itemize}
   {\bf{Acknowledgement.}} The second author would like to thank  J. Heller and P.A.{\O}stv{\ae}r for useful discussions. The authors thank to the anonymous referee for many useful comments that improved the paper.
   
\section{$RO(C_2\times \Sigma _2)$-graded Bredon cohomology of a point}
\subsection{$RO(C_2\times \Sigma _2)$-graded Bredon cohomology groups of a point}\label{rok}
The $RO(C _2\times \Sigma _2)$-graded Bredon cohomology groups  of a point with $\Z/2$-coefficients were computed in \cite{KHo}. The results were given in terms of Poincare series of graded vector spaces and reproduced also in \cite{GY} (however see the modification in Proposition \ref{caseMC1} from \cite{GY}). We have the following:
\begin{proposition} (\cite{KHo})\label{PC} Let $l,n\geq 0$ and $i,j\geq 0$. Let $\alpha$ and $\beta$ two distinct nontrivial 1-dimensional irreducible $C_2\times \Sigma_2$-representations. The Poincare series for $H^{-*+V}_{Br,K}(pt,\Z/2)$ is

a) If $V=0$ then $1$.

b) If $V=n\alpha$ then $1+x+x^2+...+x^n$.

c) If $V=-n\alpha$ then $x^{-n}+...+x^{-3}+x^{-2}$.

d) If $V=n\alpha+j\beta$ then $(1+x+...+x^n)(1+x+...+x^j)$.

e) If $V=n\alpha-j\beta$ then $(1+x+...+x^n)(x^{-j}+...+x^{-2})$.

f) If $V=-n\alpha-j\beta$ then $(x^{-n}+...+x^{-2})(x^{-j}+...+x^{-2})$.
\end{proposition}
When all three nontrivial 1-dimensional irreducible representations are involved the answer is more complicated. The following is the description of the $RO(C_2\times \Sigma_2)-$graded Bredon cohomology groups of a point in the positive cone (i.e. by definition all three nontrivial irreducible representations have nonnegative coefficients).
\begin{proposition} (\cite{KHo})\label{pc} Let $l,m,n\geq 0$. The Poincare series for $H^{-*+l\alpha+m\beta+n\gamma}_{Br,K}(pt,\Z/2)$ is 
$$(1+x+...+x^l)(1+x+...+x^m)+x(1+x+...+x^{l+m})(1+...+x^{n-1}).$$
\end{proposition}
The following is the description of the $RO(C_2\times \Sigma_2)-$graded Bredon cohomology groups of a point in the mixed cone of type I (i.e. by definition there is exactly one nontrivial irreducible $C_2\times \Sigma_2$-representation with a negative coefficient and the other two nontrivial irreducible $C_2\times \Sigma_2$-representations have nonnegative coefficients):
\begin{proposition}\cite{KHo} \label{caseMC1} Let $k,l,m\geq 1$. If $k\leq l,m$ then the Poincare series for $H^{-*+l\alpha+m\beta-k\gamma}_{Br,K}(pt,\Z/2)$ is
$$(\frac{1}{x^k}+...+\frac{1}{x})(1+x+...+x^{k-2})+x^k(1+...+x^{l-k})(1+...+x^{m-k}). $$ 
In the case $k>l$ the Poincare series for $H^{-*+l\alpha+m\beta-k\gamma}_{Br,K}(pt,\Z/2)$ is
$$\frac{1}{x^{l+1}}(1+...+x^l)(1+...+x^{l-1})+\frac{1}{x^k}(1+...+x^{k-l-2})(1+...+x^{l+m}).$$
Swapping the role of $l$ and $m$ gives the case $k>m$.
\end{proposition}
The following is the description of the $RO(C_2\times \Sigma_2)-$graded Bredon cohomology groups of a point in the mixed cone of type II (i.e. by definition there is exactly two nontrivial irreducible $C_2\times \Sigma_2$-representation with a negative coefficient and the other one nontrivial irreducible $C_2\times \Sigma_2$-representation has a nonnegative coefficient):

\begin{proposition} \cite{KHo}\label{caseMC2} Let $j,k,l\geq 1$. Then the Poincare series for $H^{-*+l\alpha-j\beta-k\gamma}_{Br,K}(pt,\Z/2)$ is 
$$\frac{1}{x^{j+k-l}}(1+...+x^{j-l-2})(1+...+x^{k-l-2})+\frac{1}{x^{l+1}}(1+...+x^l)(1+...+x^{l-1}),$$
if $j,k\geq l+1$ or 
$$\frac{1}{x^j}(1+...+x^{j-2})(1+...+x^{l-k})+\frac{1}{x^k}(1+...+x^{l-1})(1+...+x^{k-1})$$
if $l\geq k$. Swapping the role of $j$ and $k$ gives the case $l\geq j$.
\end{proposition}
The following is the description of the $RO(C_2\times \Sigma_2)-$graded Bredon cohomology groups of a point in the negative cone (i.e. by definition all three nontrivial irreducible $C_2\times \Sigma_2$- representations have negative coefficients):

\begin{proposition}\cite{KHo}\label{NC} Let $i,k,j\geq 1$. Then the Poincare series for $H^{-*-i\alpha-j\beta-k\gamma}_{Br,K}(pt,\Z/2)$ is 
$$\frac{1}{x^{i+j+k}}[(1+x+...+x^{j+k-2})(1+...+x^{i-2})+x^{i-1}(1+...+x^{k-1})(1+...x^{j-1})].$$
\end{proposition}
We conclude with the following vanishing propositions. From Proposition \ref{caseMC2} we have
\begin{proposition} \label{cvan} If $a>b\geq -q>0$ then $H^{a+q\sigma+b\epsilon+q\sigma\otimes\epsilon} _{Br,K}(pt,\Z/2)=0$.
\end{proposition}
From Proposition \ref{pc} we have
\begin{proposition} \label{cvan2} If $b,q\geq 0$ and $a \geq 1,$ then $H^{a+q\sigma+b\epsilon+q\sigma\otimes\epsilon} _{Br,K}(pt,\Z/2)=0$.
\end{proposition}
\subsection{$RO(C_2\times \Sigma _2)$-graded Bredon cohomology ring of a point}
In this section we describe the positive cone of the $RO(C_2\times \Sigma _2)$-graded Bredon cohomology ring of a point $H^{a+p\sigma+b\epsilon+q\sigma\otimes\epsilon}_{Br,K}(pt,\Z/2)$ (i.e. $p,b,q\geq 0$), as well as some cohomology classes that appear in this ring following \cite{BH}. 
\subsubsection{Positive Cone; i.e. $p,b,q\geq 0$.}
We have for $V$ an actual  $C_2\times \Sigma_2$-representation the following equality of Mackey functors:
	\begin{align*}
		\ul{\pi}_p(S^V \wedge H\ul{\Z/2})=\ul{H}^{Br,K}_p(S^V;\Z/2) \simeq \ul{H}^{Br,K}_{p-V}(pt;\Z/2) \simeq \ul{H}^{-p+V}_{Br,K}(pt;\Z/2).
	\end{align*}
	Then we have the generators of the positive cones $\Z/2[x_i,y_i]$, $i=1,2,3,$ of the $RO(C_2)$-graded Bredon cohomology rings corresponding to each of three nontrivial one-dimensional $C_2\times \Sigma_2$ representations. 
	
	We denote by $\pi^G_*$ the top level of the Mackey functor given by the equivariant stable homotopy group. The computations below follow from Proposition \ref{pc} and Proposition \ref{PC}.
	\begin{align*}
		x_1 &\in \pi_0^G(S^{0,1,0,0}\wedge H\ul{\Z/2})\simeq {H}_{Br,K}^{\sigma}(pt;\Z/2) \simeq \Z/2,\\
		y_1 &\in \pi_1^G(S^{0,1,0,0}\wedge H\ul{\Z/2})\simeq {H}_{Br,K}^{-1+\sigma}(pt;\Z/2) \simeq  \Z/2,\\
		x_2 &\in \pi_0^G(S^{0,0,1,0}\wedge H\ul{\Z/2})\simeq {H}_{Br,K}^{\epsilon}(pt;\Z/2) \simeq  \Z/2,\\
		y_2 &\in \pi_1^G(S^{0,0,1,0}\wedge H\ul{\Z/2})\simeq {H}_{Br,K}^{-1+\epsilon}(pt;\Z/2) \simeq  \Z/2,\\
		x_3 &\in \pi_0^G(S^{0,0,0,1}\wedge H\ul{\Z/2})\simeq {H}_{Br,K}^{\sigma\otimes\epsilon}(pt;\Z/2) \simeq  \Z/2,\\
		y_3 &\in \pi_1^G(S^{0,0,0,1}\wedge H\ul{\Z/2})\simeq {H}_{Br,K}^{-1+\sigma\otimes\epsilon}(pt;\Z/2) \simeq  \Z/2,\\
		\theta_1 & \in \pi^G_{-2}(S^{0,-2,0,0}\wedge H \ul{\Z/2})\simeq {H}_{Br,K}^{2-2\sigma}(pt;\Z/2) \simeq \Z/2,\\
		\theta_2 & \in \pi^G_{-2}(S^{0,0,-2,0}\wedge H \ul{\Z/2})\simeq {H}_{Br,K}^{2-2\epsilon}(pt;\Z/2) \simeq \Z/2,\\
		\theta_3 & \in \pi^G_{-2}(S^{0,0,0,-2}\wedge H \ul{\Z/2})\simeq {H}_{Br,K}^{2-2\sigma\otimes\epsilon}(pt;\Z/2) \simeq \Z/2.
	\end{align*}
with the cohomological classes given by the only non-trivial element in each of the above abelian groups. From Proposition \ref{PC} and Theorem \ref{Br}, we have that 
$$H^{*+*\sigma}_{Br,C_2}(pt,\Z/2)=\Z/2[x_1,y_1]\oplus\Z/2\left\{\frac{\theta_1}{x^{n_1}_1y^{m_1}_1}\right\},$$
and 
$$H^{*+*\epsilon}_{Br,C_2}(pt,\Z/2)=\Z/2[x_2,y_2]\oplus\Z/2\left\{\frac{\theta_2}{x^{n_1}_2y^{m_1}_2}\right\},$$
or
$$H^{*+*\sigma\otimes\epsilon}_{Br,C_2}(pt,\Z/2)=\Z/2[x_3,y_3]\oplus\Z/2\left\{\frac{\theta_3}{x^{n_1}_3y^{m_1}_3}\right\},$$
are nontrivial proper subrings in the $RO(C_2\times\Sigma_2)$-graded Bredon cohomology of a point.
	\begin{theorem} (\cite{BH})\label{pcon}
		The Mackey functor structure of the positive cone in $\ul{\pi}_\star H\ul{\Z/2}$ is given by the Mackey functor of $RO(C_2\times \Sigma_2)$-graded rings
		\begin{equation*}
			\begin{tikzcd}
				& \frac{\Z/2[x_1,y_1,x_2,y_2,x_3,y_3]}{(x_1y_2y_3+y_1x_2y_3+y_1y_2x_3)}\ar[dl] \ar[d] \ar[dr] & \\
				\frac{\Z/2[y_1,x_2,y_2,x_3,y_3]}{(x_2y_3+y_2x_3)}\ar[dr] & \frac{\Z/2[x_1,y_1,y_2,x_3,y_3]}{(x_1y_3+y_1x_3)}\ar[d] & \frac{\Z/2[x_1,y_1,x_2,y_2,y_3]}{(x_1y_2+y_1x_2)}\ar[dl]\\
				& \Z/2[y_1,y_2,y_3] & 
			\end{tikzcd}
		\end{equation*}
		where each restriction map is the identity on a generator of the domain that is also a generator of the codomain and is zero on a generator otherwise. For example, the restriction of the $x_1$ in the top level is zero in $\frac{\Z/2[y_1,x_2,y_2,x_3,y_3]}{(x_2y_3+y_2x_3)}$ and is $x_1$ in $\frac{\Z/2[x_1,y_1,y_2,x_3,y_3]}{(x_1y_3+y_1x_3)}$ and $\frac{\Z/2[x_1,y_1,x_2,y_2,y_3]}{(x_1y_2+y_1x_2)}.$ The transfer maps are always zero.
	\end{theorem}
	
To describe the other cones of the $RO(C_2\times \Sigma_2)-$graded Bredon cohomology ring of a point we need more cohomology classes than those described above.
	
The computations below follow from Proposition \ref{caseMC1}, Proposition \ref{caseMC2} and Proposition \ref{NC}.  We have the following seven new nontrivial cohomology classes:
	\begin{align*}
		\Theta &\in \pi_{-3}^G (S^{0,-1,-1,-1}\wedge H \ul{\Z/2}) \simeq {H}_{Br,K}^{3-\sigma-\epsilon-\sigma\otimes\epsilon}(pt;\Z/2) \simeq \Z/2,\\
		\kappa_1 &\in \pi_{1}^G (S^{0,-1,1,1}\wedge H \ul{\Z/2}) \simeq {H}_{Br,K}^{-1-\sigma+\epsilon+\sigma\otimes\epsilon}(pt;\Z/2) \simeq \Z/2,\\
		\kappa_2 &\in \pi_{1}^G (S^{0,1,-1,1}\wedge H \ul{\Z/2}) \simeq {H}_{Br,K}^{-1+\sigma-\epsilon+\sigma\otimes\epsilon}(pt;\Z/2) \simeq \Z/2,\\
		\kappa_3 &\in \pi_{1}^G (S^{0,1,1,-1}\wedge H \ul{\Z/2}) \simeq {H}_{Br,K}^{-1+\sigma+\epsilon-\sigma\otimes\epsilon}(pt;\Z/2) \simeq \Z/2,\\
		\iota_1 &\in \pi_{-1}^G (S^{0,1,-1,-1}\wedge H \ul{\Z/2}) \simeq {H}_{Br,K}^{1+\sigma-\epsilon-\sigma\otimes\epsilon}(pt;\Z/2) \simeq \Z/2,\\
		\iota_2 &\in \pi_{-1}^G (S^{0,-1,1,-1}\wedge H \ul{\Z/2}) \simeq {H}_{Br,K}^{1-\sigma+\epsilon-\sigma\otimes\epsilon}(pt;\Z/2) \simeq \Z/2,\\
		\iota_3 &\in \pi_{-1}^G (S^{0,-1,-1,1}\wedge H \ul{\Z/2}) \simeq {H}_{Br,K}^{1-\sigma-\epsilon+\sigma\otimes\epsilon}(pt;\Z/2) \simeq \Z/2.\\
	\end{align*}
	These cohomological classes satisfy the following relationships. For each $\{i,j,k\} = \{1,2,3\}$ (i.e. they are all distinct), we have
	\begin{align*}
		\iota_i \theta_i  &= \Theta \text{ and } \kappa_i\theta_j = \iota_k,\\
		\theta_j \theta_k &\neq 0,\\
		\iota_i \theta_j &= 0,\\
		\iota_i \kappa_i &= 0,\\
		\Theta^2 &= \theta_i \Theta = \kappa_i \Theta = \iota_i \Theta = 0,\\
		\iota_i \iota_j &= 0,\\
	\end{align*}
	and we can think of $\Theta$ as being divisible by $\theta_1,\theta_2$ and $\theta_3,$ where
	\begin{align*}
		\iota_i = \frac{\Theta}{\theta_i} \text{ and } \kappa_i  = \frac{\Theta}{\theta_j \theta_k}.
	\end{align*}
	Then we can think of $\Theta$ as being infinitely divisible by $x_i,y_i$ for $i=1,2,3$. 
		
	By degree reasons, we have
	\begin{align*}
		\Theta x_i = \Theta y_i = 0
	\end{align*}
	for all $i\in \{1,2,3\}$ and similarly
	\begin{align*}
		\iota_i x_j = \iota_i y_j = 0
	\end{align*}
	for all $i,j \in \{1,2,3\}$ with $i\neq j.$ However, it is not true that $\kappa_i x_i = \kappa_i y_i = 0$ for all $i \in \{1,2,3\}.$\\
	For each $\{i,j,k\}=\{1,2,3\}$, we have the following relations
	\begin{equation}\label{Vcomp1}
	\kappa _{i} x_i = x_j y_k + y_j x_k
	\end{equation}
	\begin{equation}\label{Vcomp2}
		\kappa_i y_i = y_j y_k,
	\end{equation}
	\begin{equation} \label{Vcomp3}	
		\kappa_i^2 \neq 0,
		\kappa_i \kappa_j = y_k^2.
	\end{equation}
	
	We cannot express $\kappa_i^2$ in terms of $x_i,y_i$ and $\theta_i.$
	
	It is proved in \cite{BH}, \cite{DV} that the entire $RO(C_2\times \Sigma_2)$-graded Bredon cohomology ring of a point can be expressed in terms of the above cohomology classes. See \cite{DV} for a more explicit description of the cohomology  classes in this ring.
\subsection{$C_2\times\Sigma_2$ topological isotropy sequence}
In this section, we prove the following vanishing theorem for the $C_2\times \Sigma_2$ topological isotropy sequence \ref{eqn:cof4}:
\begin{theorem} \label{isov}The $H^\st_{Br,K}(pt,\Z/2)$-module map induced by the $C_2\times \Sigma_2$ topological isotropy sequence \ref{eqn:cof4}
$$\tilde{H}^{a+p\sigma+b\epsilon+q\sigma\otimes\epsilon}_{Br,K}(\tilde{E}_{\Sigma_2}C_2)\rightarrow H^{a+p\sigma+b\epsilon+q\sigma\otimes\epsilon}_{Br,K}(pt,\Z/2)$$
is the zero map for all $a,b\in \mathbb{Z}$ and $p,q\geq 0$. The $H^\st_{Br,K}(pt,\Z/2)$-module map 
$$\tilde{H}^{a+p\sigma+q\sigma\otimes \epsilon}_{Br,K}(\tilde{E}_{\Sigma_2}C_2)\rightarrow H^{a+p\sigma+q\sigma\otimes\epsilon}_{Br,K}(pt,\Z/2)$$
is the zero map for all $a$ and $p,q$ not both negative or $a\leq 3$ and $p,q$ are arbitrary.

\begin{proof} Let $V=a+b\epsilon$ and take $W=a'+b'\epsilon$ with $a',b'\geq 0$ such that $V+W$ is an actual $\Sigma_2$-representation. Denote by $F(X,Y)$ the $\Sigma_2$-space of nonequivariant maps with the conjugacy action for two pointed $\Sigma_2$-spaces $X$ and $Y$. We study the map 
$$\rH^V_{Br,K}(\tilde{E}_{\Sigma_2}C_2)=[\tilde{E}_{\Sigma_2}C_2,F(S^W,M\otimes S^{V+W})]_K\rightarrow \rH^V_{Br,K}(pt_+)=[pt_+,F(S^W,M\otimes S^{V+W})]_K$$ induced by the isotropy map $pt_+\rightarrow \tilde{E}_{\Sigma_2}C_2$. But this last map factors through the map
$$pt_+\hookrightarrow S^\sigma\hookrightarrow S^{\sigma+\sigma\otimes\epsilon}\hookrightarrow \tilde{E}_{\Sigma_2}C_2$$
so the above cohomology map factors through $[S^\sigma,F(S^W,M\otimes S^{V+W})]_K$. But the target has trivial $C_2$-action so the cohomology map factors through
$[S^\sigma/C_2,F(S^W,M\otimes S^{V+W})]_K$. But $S^{\sigma}/C_2=I$, which is contractible, implying that 
$$\rH^V_{Br,K}(\tilde{E}_{\Sigma_2}C_2)=[\tilde{E}_{\Sigma_2}C_2,F(S^W,M\otimes S^{V+W})]_K\rightarrow \rH^V_{Br,K}(I)=[S^\sigma/C_2,F(S^W,M\otimes S^{V+W})]_K=0\rightarrow $$$$\rightarrow \rH^V_{Br,K}(pt_+)=[pt_+,F(S^W,M\otimes S^{V+W})]_K$$
so the topological isotropy map is zero in cohomology for indexes $a+b\epsilon$, $a,b\in \mathbb{Z}$.  The first statement of the theorem is implied by the periodicity of $\tilde{E}_{\Sigma_2}C_2$; i.e. through multiplication by the $x_1$ and $x_3$ classes from the Bredon cohomology ring of a point described above. For example, multiplication with $x_1$ gives
\[
\xymatrixrowsep{0.5in}
\xymatrixcolsep{0.5in}
\xymatrix@-0.9pc{
\rH^{a+b\epsilon}_{Br,C_2}(\tilde{E}_{\Sigma_2}C_2)\ar[r]^{0}\ar[d]^{\stackrel{x_1}{\simeq}}&
H^{a,b} _{Br,C_2}(pt,\Z/2)\ar[d]^{x_1}
\\ 
\rH^{a+\sigma+b\epsilon}_{Br,K}(\tilde{E}_{\Sigma_2}C_2)\ar[r]&
H^{a+\sigma +b\epsilon} _{Br,K}(pt,\Z/2).
}
\]
The last case of the theorem is symmetric in $p$ and $q$. Suppose for example that $p=-n<0$ and $q\geq 0$. By multiplication with $x_3$ as above we can assume that $q=0$. Let $a>3$ (by periodicity and the fact that $\rH^{a+0\sigma+0\epsilon+0\sigma\otimes\epsilon}_{Br,K}(\tilde{E}_{\Sigma_2}C_2)=0$ if $a\leq 3,$ we conclude that $\rH^{a+p\sigma+q\sigma\otimes\epsilon}_{Br,K}(\tilde{E}_{\Sigma_2}C_2)=0$ for any $p,q\in \mathbb{Z},$ so the map is zero). Then
$$\rH^{a-n\sigma}_{Br,C_2}(\tilde{E}_{\Sigma_2}C_2)=[\tilde{E}_{\Sigma_2}C_2,F(S^{n\sigma},M\otimes S^{a})]_K\rightarrow \rH^V_{Br,K}(pt_+)=[pt_+,F(S^{n\sigma},M\otimes S^{a})]_K.$$ 
We describe the action of $C_2\times \Sigma_2$ on $S^{\sigma+\sigma\otimes\epsilon}=S(1\oplus\sigma\oplus\sigma\otimes \epsilon)$. Denote by $a$ the nontrivial element of $C_2$ and $b$ the nontrivial element  of $\Sigma_2$. Then $ab$ is the nontrivial element of the diagonal subgroup $\Delta$. Then $a(x,y,z)=(x,-y,-z)$, $b(x,y,z)=(x,y,-z)$ and $ab(x,y,z)=(x,-y,z)$. When we restrict to the action on $S^{\sigma\otimes\epsilon}=S(1\oplus\sigma\otimes\epsilon)$ then $a(x,0,z)=(x,0,-z)$, $b(x,y,z)=(x,0,-z)$ and $ab(x,0,z)=(x,0,z)$. In conclusion $\Delta$ acts trivially on $S^{\sigma\otimes\epsilon}$. Also $$I=S^{\sigma\otimes\epsilon}/C_2=S^{\sigma\otimes\epsilon}/\Sigma_2=S^{\sigma\otimes\epsilon}/K$$ because $C_2$ and $\Sigma_2$ act in the same way on $S^{\sigma\otimes\epsilon}$. Because $\Sigma_2$ acts trivially on $F(S^{n\sigma},M\otimes S^{a})$ it implies that $C_2$ acts as $\Delta$ on this space. 

Now the map $pt_+\rightarrow \tilde{E}_{\Sigma_2}C_2$ factors as 
$$pt_+\hookrightarrow S^{\sigma\otimes \epsilon}\hookrightarrow S^{\sigma\oplus\sigma\otimes\epsilon}\hookrightarrow \tilde{E}_{\Sigma_2}C_2$$
and consequently we have a factorization of the topological isotropy map 
$$\rH^{a-n\sigma}_{Br,C_2}(\tilde{E}_{\Sigma_2}C_2)\rightarrow [S^{\sigma\otimes\epsilon},F(S^{n\sigma},M\otimes S^{a})]_K\rightarrow \rH^{a-n\sigma}_{Br,C_2}(pt_+).$$
But the homotopy group of pointed maps is zero i.e.$$[S^{\sigma\otimes\epsilon},F(S^{n\sigma},M\otimes S^{a})]_K=[S^{\sigma\otimes\epsilon}/\Sigma_2,F(S^{n\sigma},M\otimes S^{a})]_K=[I,F(S^{n\sigma},M\otimes S^{a})]_K=0.$$
It implies the second statement of the theorem.
\end{proof}
\end{theorem}
\subsection{$\Sigma_2$-equivariant topological spaces}
Consider the $\Sigma_2$-equivariant topological space
$$B_{\Sigma_2}C_2=E_{\Sigma_2}C_2/C_2=colim_n S(n\sigma+n\sigma\otimes\epsilon)/C_2.$$
It is the realization of $$\BG C_2=\EG C_2/C _2=colim _n (\mathbb{A}(n\sigma)\setminus{0})/C_2,$$ the classifying space of $C_2$ over the field of real numbers.
The $RO(\Sigma_2)$-graded Bredon cohomology of $B_{\Sigma_2}C_2$ is given below:
\begin{theorem} (\cite{Kr},\cite{HK1}) \label{comp1} We have that $$H^{*+*\epsilon} _{Br,C_2}(B _{\Sigma _2}C _2,\Z/2)=H^{*+*\epsilon} _{Br,C_2}(pt,\Z/2)[c,b]/(c^2=x_2 c+y_2b)$$ where $|c|=\epsilon$ and $|b|=1+\epsilon$ and $x_2,y_2\in H^{\st}_{Br,C_2}(pt,\Z/2)$ are the usual classes in the degrees $\epsilon$ and $\epsilon-1$ respectively. 
\end{theorem}
The motivic cohomology of $\BG C_2$ over the field of real numbers is computed below:
\begin{theorem} (\cite{Voc}) \label{comp2} We have that 
 $$H^{*,*}(\mathbf{B}C _2,\Z/2)=H^{*,*}(\R,\Z/2)[s,t]/(s^2=\tau t+\rho s)$$ where $|s|=(1,1)$ and $|t|=(2,1)$, $\tau\in H^{0,1}(\R,\Z/2)=\Z/2$ and $\rho$ is the class of $[-1]\in H^{1,1}(\R,\Z/2)=\Z/2$. 
 \end{theorem}
 
 It is obvious that the realization map 
$$H^{a,b}(\mathbf{B}C _2,\Z/2)\to H^{a-b+b\epsilon} _{Br,C_2}(B _{\Sigma _2}C _2,\Z/2)$$ is an isomorphism for any $a\in \Z$ such that  $a\leq 2b$. The realization map sends
$$s\rightarrow c,$$
$$t\rightarrow b,$$
$$\tau\rightarrow y_2,$$
$$\rho\rightarrow x_2.$$
The way we choose the generator $s$ in Theorem \ref{comp2} is the following: it is the unique element $s\in H^{1,1}(\mathbf{B}C_2,\Z/2)$ such that the restriction to $H^{1,1}(pt,\Z/2)$ is zero and the Bockstein homomorphism $\delta: \rH^{1,1}(-,\Z/2)\rightarrow \rH^{2,1}(-,\Z)$ sends $\delta(s)=t$. We choose $t$ to be the Euler class of the line bundle on $\BG C_2$ corresponding to the sign representation of $C_2$.
 
\begin{lemma}
	\label{gap}
	We have that
	\begin{align*}
		H^{a+b\epsilon}_{Br,C_2}(B_{\Sigma_2}C_2,\Z/2) = 0
	\end{align*}
	when $b= a-1$ or $b= a-2.$
\end{lemma}
\begin{proof}
	We know that
	\begin{align*}
		H^{a+b\epsilon}_{Br,C_2}(B_{\Sigma_2}C_2,\Z/2)  = \frac{H^{a+b\epsilon}_{Br,C_2}(pt,\Z/2) [c,b]}{\left(c^2=y_2b+x_2c\right)},
	\end{align*}
	so every generator can be obtained by multiplying an element $ \alpha \in H^\st_{Br,C_2}(pt,\Z/2)$ with either $b^n$ or $cb^n.$ If $\alpha$ is in the positive cone, the possibilities are bounded by the products $b^n$ along the line $a=b,$ and if $\alpha$ is in the negative cone, the possibilities are bounded by the products $\theta_2 c b^n,$ along the line $b=a-3.$ This is shown in the diagram below.
\begin{figure}[H]

\begin{tikzpicture}
				[
				dot/.style={circle,draw=black, fill,inner sep=1pt},
				]
				
				\foreach \y in {0, ..., 4}
				\foreach \x in {0, ..., \y} {
					\node[dot] at (-\x,\y){};
				}
				
				\foreach \y in {0, ..., 2}
				\foreach \x in {0, ..., \y} {
					\node[dot] at (\x+2,-\y-2){};
				}
				
				\node[dot] at (2,-1){};
				
				\foreach \x in {-4,...,4}
				\draw (\x,-5.1) -- node[below,yshift=-1mm] {\x} (\x,-4.9);

				\foreach \y in {-4,...,4}
				\draw (-4.9,\y) -- node[below,yshift=2.5mm,xshift=-3mm] {\y} (-5.1,\y);

				\draw[->,thick,-latex] (-5,-5.3) -- (-5,5);
				\draw[->,thick,-latex] (-5.3,-5) -- (5,-5);
				
				\draw[->,thick,-latex] (0,0) -- (4,4);
				\draw[->,thick,-latex] (2,-1) -- (5,2);
				
				\draw[->,dashed,-latex] (-2,-3) -- (5,4);
				\draw[->,dashed,-latex] (-2,-4) -- (5,3);
				
				\node[below,xshift=4mm,yshift=1mm] at (5,2) {$\theta_2 cb^n$};
				\node[below,xshift=3mm,yshift=1mm] at (4,4) {$b^n$};

				\node[below,yshift=-1mm] at (5,-5) {$a$};
				\node[below,xshift=-2mm] at (-5,5) {$b$};  
				
				\node[below,xshift=3mm,yshift=3mm] at (2,5) {${\displaystyle H_{Br,C_2}^{a,b}(B_{\Sigma_2}C_2;\Z_2)} $};

				\node[below,xshift=3mm,yshift=1mm] at (0,0) {$1$};
				
				\node[below,xshift=3mm,yshift=2mm] at (0,1) {$x_2$};
				\node[below,xshift=3mm,yshift=2mm] at (-1,1) {$y_2$};
				
				\foreach \x in {2,...,4}
				\node[below,xshift=3mm,yshift=3mm] at (0,\x) {$x_2^\x$};

				\node[below,xshift=5mm,yshift=1mm] at (-1,2) {$y_2x_2$};
				\foreach \x in {2,...,3}
				\node[below,xshift=5mm,yshift=3mm] at (-1,\x+1) {$y_2x_2^\x$};
				
				\node[below,xshift=3mm,yshift=3mm] at (-2,2) {$y_2^2$};
				\node[below,xshift=5mm,yshift=3mm] at (-2,3) {$y_2^2x_2$};
				\node[below,xshift=5mm,yshift=3mm] at (-2,4) {$y_2^2x_2^2$};
				
				\node[below,xshift=3mm,yshift=3mm] at (-3,3) {$y_2^3$};
				\node[below,xshift=5mm,yshift=3mm] at (-3,4) {$y_2^3x_2$};
				
				\node[below,xshift=3mm,yshift=3mm] at (-4,4) {$y_2^4$};
				
				
				\node[below,xshift=4mm,yshift=3mm] at (2,-1) {$\theta_2 c$};
				
				\node[below,xshift=3mm,yshift=3mm] at (2,-2) {$\theta_2$};
				
				\node[below,xshift=3mm,yshift=3mm] at (2,-3) {$\frac{\theta_2}{x_2}$};
				\node[below,xshift=3mm,yshift=3mm] at (2,-4) {$\frac{\theta_2}{x_2^2}$};
				
				\node[below,xshift=3mm,yshift=3mm] at (3,-3) {$\frac{\theta_2}{y_2}$};
				\node[below,xshift=4mm,yshift=3mm] at (3,-4) {$\frac{\theta_2}{y_2x_2}$};
				
				\node[below,xshift=3mm,yshift=3mm] at (4,-4) {$\frac{\theta_2}{y_2^2}$};
				

			\end{tikzpicture}
		\caption{}
		\label{TEST3}
	\end{figure}
	
\end{proof}
\begin{corollary}\label{comp3} If $a\leq b+2, b< 0$ then $$H^{a,b} _{Br,C_2}(B _{\Sigma _2}C _2,\Z/2)=0.$$ 
If $b=0$ then $$\rH^{a} _{Br,C_2}(B_{\Sigma_2}C_2,\Z/2)=0$$ if $a\leq 2$.
\end{corollary}

In conclusion we have
\begin{corollary} \label{cr} The realization maps are isomorphisms
$$H^{a,b}(\mathbf{B}C _2,\Z/2)\simeq H^{a-b+b\epsilon} _{Br,C_2}(B _{\Sigma _2}C _2,\Z/2),$$ for any $a,b\in \Z$ such that  $a\leq 2b$.
\end{corollary}

We have that
 $$H^{*,*}(\mathbf{B}C _2,\Z/2)\simeq \Z/2[\tau,\rho][s,t]/(s^2=\tau t+\rho s),$$
 because the motivic cohomology of the real numbers agrees with the positive cone of the $RO(C _2)$-graded Bredon cohomology of a point (see Proposition \ref{classic}).
 
  Looking to Figure \ref{TEST3}, we see that this ring is represented by elements in the upper cone starting from the diagonal line given by the powers of $b$. The lower cone that is in the lower half plane $a>b$ is not included in the image of the realization maps from Corollary \ref{cr}.
 
 Considering $W_q=A(q\sigma)\setminus\{0\}$, $q>0$, which has a free $C_2$-action, we have that $H^{*,*}_{C_2}(W_q,\Z/2)$ is generated over the motivic cohomology of a point by $st^i,t^i$ with $0\leq i\leq q-1$ because of \cite{Voc} and Proposition \ref{v2n}.  Thus we have  over a field $k$ of characteristic zero
 $$H^{*,*}_{C_2}(W_q,\Z/2)\simeq H^{*,*}(k)[s,t]/(s^2=\tau t+\rho s,t^q)$$
 and the target of the realization maps over reals is in
 $$H^{*,*}_{Br,C_2}(P(\R^{q+q\sigma}),\Z/2)\simeq H^{*,*}_{Br,C_2}(pt)[c,b]/(c^2=x_2c+y_2b,b^q=0)$$
 computed in [\cite{Kr},Theorem 4.11]. In conclusion the realization maps are isomorphisms  
 $$H^{a,b} _{C_2}(W_q,\Z/2)\simeq H^{a-b+b\epsilon} _{Br,C_2}(S(q\sigma+q\sigma\otimes\epsilon)/C_2,\Z/2),$$ for any $a,b\in \Z$ such that  $a\leq 2b$. 
  \begin{figure}[H]
 	
 	\begin{tikzpicture}
 		[
 		dot/.style={circle,draw=black, fill,inner sep=1pt},
 		]
 		
 		\foreach \y in {0, ..., 4}
 		\foreach \x in {0, ..., \y} {
 			\node[dot] at (-\x,\y){};
 		}

 		\foreach \x in {-4,...,4}
 		\draw (\x,-0.1) -- node[below,yshift=-1mm] {\x} (\x,0.1);

 		\foreach \y in {0,...,4}
 		\draw (-4.9,\y) -- node[below,yshift=2.5mm,xshift=-3mm] {\y} (-5.1,\y);

 		\draw[->,thick,-latex] (-5,0) -- (-5,5);
 		\draw[->,thick,-latex] (-5,0) -- (5,0);
 		
 		\draw[->,thick,-latex] (0,0) -- (4,4);

 		\node[below,xshift=3mm,yshift=1mm] at (4,4) {$b^{q-1}$};

 		\node[below,yshift=-1mm] at (5,0) {$a$};
 		\node[below,xshift=-2mm] at (-5,5) {$b$};  
 		
 		\node[below,xshift=3mm,yshift=3mm] at (2,5) {${\displaystyle Re(H^{*,*}_{C_2}(W_q,\Z/2))\subset H^{*,*}_{Br,C_2}(P(\R^{q+q\sigma}),\Z/2)}$};

 		\node[below,xshift=1mm,yshift=6mm] at (0,0) {$1$};
 		
 		\node[below,xshift=3mm,yshift=2mm] at (0,1) {$x_2$};
 		\node[below,xshift=3mm,yshift=2mm] at (-1,1) {$y_2$};
 		
 		\foreach \x in {2,...,4}
 		\node[below,xshift=3mm,yshift=3mm] at (0,\x) {$x_2^\x$};

 		\node[below,xshift=5mm,yshift=1mm] at (-1,2) {$y_2x_2$};
 		\foreach \x in {2,...,3}
 		\node[below,xshift=5mm,yshift=3mm] at (-1,\x+1) {$y_2x_2^\x$};
 		
 		\node[below,xshift=3mm,yshift=3mm] at (-2,2) {$y_2^2$};
 		\node[below,xshift=5mm,yshift=3mm] at (-2,3) {$y_2^2x_2$};
 		\node[below,xshift=5mm,yshift=3mm] at (-2,4) {$y_2^2x_2^2$};
 		
 		\node[below,xshift=3mm,yshift=3mm] at (-3,3) {$y_2^3$};
 		\node[below,xshift=5mm,yshift=3mm] at (-3,4) {$y_2^3x_2$};
 		
 		\node[below,xshift=3mm,yshift=3mm] at (-4,4) {$y_2^4$};

 	\end{tikzpicture}
 	\caption{}
 	\label{TEST6}
 \end{figure}

 In conclusion, the Bredon motivic cohomology in integer indexes of $W_q$ is represented in Figure \ref{TEST3} as the truncation of the upper half plane along the line $(q,b)$ with $b\geq 0$ with the realization described in Figure \ref{TEST6}. 
 
 The following proposition implies $t^q=0$ for dimension reasons. Notice that if $k<n$ then $t^k\in H^{2k,k}_{C_2}(W_q,\Z/2)=H^{2k,k}_{C_2}(\BG C_2,\Z/2)$ is nonzero from \cite{Voc}.
  \begin{proposition} \label{v2n} $$H^{2b,b}_{C_2}(W_{-q},\Z/2)=0$$ for $b\geq -q>0$ and any field $k$ of characteristic zero. 
  \end{proposition}
  \begin{proof} Because $W_{-q}$ has free $C_2$-action we have that $$H^{2b,b}_{C_2}(W_{-q},\Z/2)\simeq H^{2b,b}(W_{-q}/C_2,\Z/2)=0$$
  from the vanishing bounds of motivic cohomology because $2b>b+(-q)=b+dim(W_{-q}/C_2)$ (see \cite{MVW}). 
 \end{proof}
\section{The case $b+q<0.$}
In this section we discuss the realization maps from the region $b+q<0$ of Figure \ref{TEST1}. The main results of the section are given in Corollary \ref{itilde} and Theorem \ref{ncinj}.

The following proposition shows that in the region $b+q<0$ the Bredon motivic cohomology of real numbers identifies with  the reduced Bredon motivic cohomology of $\wt\EG C_2$.
\begin{proposition} \label{tilde} If $b\geq 0$ and $b+q<0$ then $H^{a+p\sigma,b+q\sigma} _{C _2}(\R,\Z/2)\simeq \rH^{a+p\sigma,b+q\sigma} _{C _2}(\wt\EG C _2,\Z/2)$.
\end{proposition}
\begin{proof} The proposition follows from the motivic isotropy sequence and the fact that $$H^{a+p\sigma,b}_{C_2}(\EG C _2,\Z/2)=0$$ if $b<0$. This follows from the fact that $H^{a,b}(\EG C _2,\Z/2)=0$ if $b<0$ (from the vanishing of motivic cohomology in this range) and inductively, from the long exact sequence induced in the Bredon motivic cohomology of $\EG C_2$ by the cofiber sequence $C_{2+}\rightarrow S^0\rightarrow S^\sigma$.

We also use the periodicity of the Bredon motivic cohomology of $\EG C_2$ to obtain $$H^{a+p\sigma,b+q\sigma}_{C _2}(\EG C _2,\Z/2)\simeq H^{a+2q+(p-2q)\sigma,b+q}_{C _2}(\EG C _2,\Z/2)=0$$ if $b+q<0$. 
\end{proof}
\begin{proposition}\label{vanr} ([\cite{HOV2},Prop. 2.8]) If $b\leq 0$ and $b+q<0$ then $H^{a+p\sigma,b+q\sigma}_{C _2}(\R,\Z/2)=0$.
\end{proposition} 
\begin{proof}
This follows from the motivic isotropy sequence and the known periodicities. This proposition is true for any $C _2$-equivariant smooth scheme $X$.
\end{proof}
The following proposition computes the realization maps of the Bredon motivic cohomology of $\wt\EG C_2$.
\begin{proposition} \label{x} a) The realization maps  $$\rH^{a,b} _{C _2}(\wt\EG C _2,\Z/2)\simeq \rH^{a-b+b\epsilon} _{Br,C_2}(\tilde{E} _{\Sigma_2} C _2,\Z/2)$$ are isomorphisms for $a\leq 2b+1$ and any $b\in \Z$. Moreover $$\rH^{a,b} _{C _2}(\wt\EG C _2,\Z/2)= 0$$ if $a>2b+1$ or $b\leq 0$.
 
 b) The region $b+q<0$, $b\geq 1$ of Figure \ref{TEST1} is a $ \MMt$-submodule of $$\rH^{*,*} _{C _2}(\wt\EG C _2,\Z/2)\simeq \rH^{*,*} _{C _2}(\Sigma^1\BG C _2,\Z/2)[\sigma^{\pm1},\epsilon^{\pm1}]$$ with invertible cohomology classes $\sigma$ of degree $(\sigma,0)$ and $\epsilon$ of degree $(0,\sigma)$. 
 \end{proposition}
 \begin{proof}
 We have that $\BG C _2\rightarrow pt$ admits a section, which makes the long exact sequence induced by the motivic isotropy sequence in integer indexes to split. This gives short exact sequences
 $$0\rightarrow H^{*,*} _{C _2}(pt,\Z/2)\rightarrow H^{*,*} _{C _2}(\BG C _2,\Z/2)\rightarrow \rH^{*+1,*} _{C _2}(\wt\EG C _2,\Z/2)\rightarrow 0.$$
 This gives an isomorphism $\rH^{*,*}_{C_2}(\wt\EG C _2,\Z/2)\simeq \rH^{*,*}_{C_2}(\Sigma^1\BG C _2,\Z/2)$.  
 We also have a commutative diagram
 \[
\xymatrixrowsep{0.3in}
\xymatrixcolsep{0.25in}
\xymatrix@-0.9pc{
\rH^{a+1,b} _{C _2}(\wt\EG C _2,\Z/2)\ar[r]^{\simeq}\ar[d]&
\rH^{a,b} _{C _2}(\BG C _2,\Z/2)\ar[d]
\\ 
\rH^{a+1-b+b\epsilon} _{Br,C_2}(\tilde{E} _{\Sigma _2}C _2,\Z/2)\ar[r]^{\simeq}&
\rH^{a-b+b\epsilon} _{Br,C_2}(B _{\Sigma _2}C _2,\Z/2)
}
\]
where the right vertical map is an isomorphism if $a\leq 2b$ (see Corollary \ref{cr}). When $a\geq 2b+2$ or $b\leq 0$ we have that $$\rH^{a,b} _{C _2}(\wt\EG C _2,\Z/2)= 0$$ from [\cite{HOV2}, Prop. 2.7 and Prop. 2.9].

According to Proposition \ref{tilde}, we have that the groups in the region $b+q<0$, $b\geq 1$ of Figure \ref{TEST1} are isomorphic to the reduced Bredon motivic cohomology of $\wt\EG C_2$. It is obvious that $$\rH^{*,*} _{C _2}(\wt\EG C _2,\Z/2)\simeq \rH^{*,*} _{C _2}(\Sigma^1\BG C _2,\Z/2)[\sigma^{\pm1},\epsilon^{\pm1}]$$ according to the discussion above and the periodicities of the latest cohomology.
 \end{proof}
We know that over any field $k$ of characteristic zero we have $\rH^{a,b}_{C_2}(\wt\EG C _2,\Z/2)=0$ for $a\leq 1$ [\cite{HOV1}, Lemma 4.2]. 

 We discuss now the relation between the realization maps of Bredon motivic cohomology of real numbers and the realization maps of Bredon motivic cohomology of $\wt\EG C_2$.

 We have for $a\leq 2b+1$,  $b+q<0$ the diagram
  \[
\xymatrixrowsep{0.3in}
\xymatrixcolsep{0.25in}
\xymatrix@-0.9pc{
\rH^{a+p\sigma,b+q\sigma} _{C _2}(\wt \EG C _2,\Z/2)\ar[r]^{\simeq}\ar[d]^{\simeq}&
H^{a+p\sigma,b+q\sigma} _{C _2}(\R,\Z/2)\ar[d]
\\ 
\rH^{a-b+(p-q)\sigma+b\epsilon+q\sigma\otimes\epsilon} _{Br,K}(\tilde{E} _{\Sigma _2}C _2,\Z/2)\ar[r]&
H^{a-b+(p-q)\sigma+b\epsilon+q\sigma\otimes\epsilon} _{Br,K}(pt,\Z/2)
}
\]
If $a>2b+1$, $b+q<0$ then the upper horizontal map is zero because it is an isomorphism of two groups that are zero. It implies that the right vertical realization maps are trivial monomorphisms (the domain is zero) in this case. Also, in this case, the left vertical realization maps in the diagram above are monomorphisms, but  not necessary isomorphisms. 

In the case $a\leq 2b+1$, $b+q<0$ the right vertical realization maps coincide with an induced map in the $C_2\times\Sigma_2$ topological isotropy sequence (the bottom horizontal map in the above diagram). 

 According to the periodicity of the Bredon cohomology of $\tilde{E} _{\Sigma _2}C _2$ and the proof of Proposition \ref{x} we have 
$$ \rH^{a-b+(p-q)\sigma+b\epsilon+q\sigma\otimes\epsilon} _{Br,K}(\tilde{E} _{\Sigma _2}C _2,\Z/2)\simeq \rH^{a-b+b\epsilon} _{Br,C_2}(\tilde{E} _{\Sigma _2}C _2,\Z/2)\simeq \rH^{a-b-1+b\epsilon} _{Br,C_2}(B _{\Sigma _2}C _2,\Z/2).$$
The latest group is reviewed in Theorem \ref{comp1}. From Corollary \ref{comp3} we conclude that $$\rH^{a+p\sigma+q\sigma\otimes \epsilon} _{Br,K}(\tilde{E}_{\Sigma_2}C_2,\Z/2)=0$$ if $a\leq 3$ and any $p,q\in \mathbb{Z}$.

In conclusion, we have the following corollary about the realization maps in the region $b+q<0,$ $b\geq 1$ of Figure \ref{TEST1}:
\begin{corollary} \label{itilde} Let $b+q<0$, $p\in \Z$. The realization maps 
$$H^{a+p\sigma,b+q\sigma}_{C_2}(\R ,\Z/2)\rightarrow H^{a-b+(p-q)\sigma+b\epsilon+q\sigma\otimes\epsilon}_{Br,K}(pt,\Z/2)$$ 
are trivial monomorphisms (the domain is zero) if $a>2b+1$ or $a\leq 1$ or $b\leq 0$. When $a\leq 2b+1$, the displayed realization maps coincide with the $C_2\times \Sigma_2$-isotropy sequence maps $$\rH^{a-b+(p-q)\sigma+b\epsilon+q\sigma\otimes\epsilon} _{Br,K}(\tilde{E} _{\Sigma _2}C _2,\Z/2)\rightarrow H^{a-b+(p-q)\sigma+b\epsilon+q\sigma\otimes\epsilon} _{Br,K}(pt,\Z/2).$$
\end{corollary}
The next theorem says that all realization maps from Corollary \ref{itilde} are monomorphisms. 
\begin{theorem}
	\label{ncinj}
	The connecting map in the $C_2\times \Sigma_2$ topological isotropy sequence 
	\begin{align*}
		\dl:	H_{Br,K}^{a+p\sigma+b\epsilon+q\sigma\otimes\epsilon} (E_{\Sigma_2}C_2)\to \hr_{Br,K}^{a+1+p\sigma+b\epsilon+q\sigma\otimes\epsilon}(\tilde{E}_{\Sigma_2}C_2)
	\end{align*}
	vanishes when both $1-b\leq a\leq b+1$ and $b+q<0.$
\end{theorem}
\begin{proof}
	We know that all the generators in $H^\st_{Br,K}(E_{\Sigma_2}C_2)$ are of the form
	\begin{align*}
		\al =	\fc{\ta_2}{x_2^{n_2}y_2^{m_2}}x_1^{n_1}x_3^{n_3}y_1^{m_1}y_3^{m_3} \kappa_2^n \txo {x_2^{n_2}y_2^{m_2}}x_1^{n_1}x_3^{n_3}y_1^{m_1}y_3^{m_3} \kappa_2^n,
	\end{align*}
	where $n_i,m_i \geq 0$ and $n \in \Z.$ This is because, using Theorem \ref{isov} and the $C_2\times \Sigma_2$ topological isotropy sequence, one has
	$$H^\star _{Br,K}(E_{\Sigma_2}C_2)=\frac{H^{*+*\epsilon}_{Br,C_2}(pt)[x_3,y_3,x_1,y_1,\kappa^{\pm 1}_2]}{(\kappa_2y_2=y_1y_3,\kappa_2x_2=x_1y_3+x_3y_1)}$$
	as in [\cite{DV}, Theorem 4.3.2]. 
	
	We have from the $C_2\times \Sigma_2$ topological isotropy sequence that if $n>0$
	$$H^{n-n\sigma+n\epsilon-n\sigma\otimes\epsilon}_{Br,K}(E_{\Sigma_2}C_2,\Z/2)=\Z/2(\kappa_2^{-n})\stackrel{\delta}{\simeq} \rH^{n+1-n\sigma+n\epsilon-n\sigma\otimes\epsilon}_{Br,K}(\tilde{E}_{\Sigma_2}C_2,\Z/2)=\Z/2\left(\frac{\Sigma b^n}{x^n_1x^n_3}\right),$$
	so $\delta(\kappa_2^{-n})=\frac{\Sigma b^n}{x^n_1x^n_3}$. The isomorphism follows from the fact that $H^{n+1-n\sigma+n\epsilon-n\sigma\otimes\epsilon} _{Br,K}(pt,\Z/2)=0$ (see Proposition \ref{caseMC2}). 
	
	If $n \geq 0,$ then $\al \in \ker \dl$ from Theorem \ref{isov}. This is because it belongs to the image of the injective map (in this range) given by $H^\st _{Br,K}(pt)\rightarrow H^\st _{Br,K}(E_{\Sigma_2}C_2,\Z/2)$.

	Let $n= -N$ where $N > 0.$ If
	\begin{align*}
		\al =	\fc{\ta_2}{x_2^{n_2}y_2^{m_2}}x_1^{n_1}x_3^{n_3}y_1^{m_1}y_3^{m_3} \kc{N}\in H^{a+p\sigma+b\epsilon+q\sigma\otimes\epsilon}_{Br,K}(E_{\Sigma_2}C_2,\Z/2)
	\end{align*}
	then, because $\delta$ is a $H^{\star}_{Br,K}(pt,\Z/2)$-module map, we have
	\begin{align*}
		\dl(\al) &= \fc{\ta_2}{x_2^{n_2}y_2^{m_2}}x_1^{n_1}x_3^{n_3}y_1^{m_1}y_3^{m_3} \fc{\susp b^N}{x_1^Nx_3^N}\\
		&= x_1^{n_1-N}x_3^{n_3-N} \cp y_1^{m_1}y_3^{m_3} \cp \left(\fc{\ta_2}{x_2^{n_2}y_2^{m_2}}\susp b^N\right).
	\end{align*}
	If $\al$ lives in the range $1-b\leq a\leq b+1$, then we must have that
	\begin{align*}
		y_1^{m_1}y_3^{m_3} \cp \left(\fc{\ta_2}{x_2^{n_2}y_2^{m_2}}\susp b^N\right) = 0,
	\end{align*}
	since multiplication with $y_1$ and $y_3$ must eventually cross the gap along the line $b=a-3$ in $\hr^\st_{Br,K}(\et).$ Indeed, we have that
	\begin{align*}
		\left|\fc{\ta_2}{x_2^{n_2}y_2^{m_2}}\susp b^N\right|=(A,0,B,0)=(N+m_2+3,0,N-m_2-n_2-2,0),
	\end{align*}
	where $B \leq A-5< A-3.$ But we know that $\al$ lives in the range $a\leq b+1,$ so $A-3 -m_1-m_3<a= A-m_1-m_3-1 \leq b+1=B+1.$ Then we have
	\begin{align*}
		A-3 > B \txa A-3 \leq  B+ m_1+m_3,
	\end{align*}
	so $A-m_1'- m_3' -3= B$ for some $m_1' \in [0,m_1]$ and $m_3' \in [0,m_3].$ Then
	
	\begin{align*}
		y_1^{m_1'}y_3^{m_3'} \cp \left(\fc{\ta_2}{x_2^{n_2}y_2^{m_2}}\susp b^N\right) \in \hr^{A-m_1'- m_3'+m_1'\sigma+B\epsilon+m_3'\sigma\otimes\epsilon}_{Br,K}(\et)=0,
	\end{align*}
	since by Lemma \ref{gap} $\hr^\st_{Br,K}(\et)=0$ along the line $a-3 =b.$ Thus,
	\begin{align*}
		\dl(\al) = x_1^{n_1-N}x_3^{n_3-N} \cp y_1^{m_1}y_3^{m_3} \cp \left(\fc{\ta_2}{x_2^{n_2}y_2^{m_2}}\susp b^N\right) = x_1^{n_1-N}x_3^{n_3-N} \cp y_1^{m_1-m_1'}y_3^{m_3-m_3'}y_1^{m_1'}y_3^{m_3'} \cp \left(\fc{\ta_2}{x_2^{n_2}y_2^{m_2}}\susp b^N\right)=0.
	\end{align*}
	If 
	\begin{align*}
		\al =	{x_2^{n_2}y_2^{m_2}}x_1^{n_1}x_3^{n_3}y_1^{m_1}y_3^{m_3} \kc{N},
	\end{align*}
	then the condition $b+q <0$ implies that
	\begin{align*}
		n_2+m_2+n_3+m_3 < 0,
	\end{align*}
	which is false. Hence there are no non-vanishing $\al$ of this form in the given range.
\end{proof}
\begin{remark} We use the vanishing result of Proposition \ref{ncinj} in the last section, but we can notice from the above proof that there is also a vanishing for the larger range $a\leq b+2$, $b+q<0$. The fact that there is no lower bound for $a$ in Proposition \ref{ncinj} follows also from Corollary \ref{itilde}. 
\end{remark}
\section{The case $b\geq 0$, $b+q\geq 0$.}
In this section, we will prove that all the realization maps in the range $b\geq 0$, $b+q\geq 0$ are isomorphisms. This is Theorem \ref{posisom} below and it is essentially based on the similar result in indexes $(a+2q\sigma,b+q\sigma)$ proved in the Theorem \ref{2q}.

Notice that this is a generalization of the fact that the motivic cohomology of the real numbers has isomorphic realization maps into the $RO(C_2)-$graded Bredon cohomology of a point when the weight is greater or equal to zero (Proposition \ref{classic}).

\begin{theorem} \label{2q} We have that if $b\geq 0$, $b+q\geq 0$ then
$$H^{a+2q\sigma,b+q\sigma}_{C _2}(\R,\Z/2)\simeq H^{a-b+q\sigma+b\epsilon+q\sigma\otimes \epsilon}_{Br,K}(pt,\Z/2).$$

If moreover $a>2b\geq 0$ then $H^{a+2q\sigma,b+q\sigma}_{C _2}(\R,\Z/2)=0$ and  the codomain is also zero in this range for any $q\in \Z$, $b+q\geq 0$.
\end{theorem}

\begin{proof}
 We have that for an actual $C _2$-representation $V=b+q\sigma$, $b,q\geq 0$ (see [\cite{V}, Proposition 3.4 and Proposition 3.5]), $$H^{a+2q\sigma,b+q\sigma} _{C _2}(\R,\Z/2)=H^{a-2b} _{GNis}(\R,C _*z(V)^{\Z/2}).$$

 From \cite{Nie} we have a decomposition in motivic complexes in $DM^{-}(\R)$
 $$C _*z(V)^{\Z/2}\simeq \oplus  _{j=0}^{n-1}(\Z/2(j)[2j]\oplus\Z/2(j)[2j+1])\oplus \Z/2(n)[2n].$$
 Applying cohomology as above we obtain that 
 $$H^{a+2q\sigma,b+q\sigma} _{C _2}(\R,\Z/2)\simeq \oplus _{j=0}^{q-1} H^{a+2j,j+b}(\R,\Z/2)\oplus H^{a+2j+1,j+b}(\R,\Z/2)\oplus H^{a+2q,q+b}(\R,\Z/2).$$ In particular, if $a>b\geq 0$, $q\geq 0$  then 
 $H^{a+2q\sigma,b+q\sigma} _{C _2}(\R,\Z/2)=0$; the codomain of its realization map is zero in this case from Proposition \ref{cvan2}.
 
 Applying the realization functor to the above decomposition of complexes we obtain a decomposition in the topological target, giving a decomposition for the  $RO(C_2\times \Sigma_2)$-graded Bredon cohomology of a point:
 $$H^{a-b+q\sigma+b\epsilon+q\sigma\otimes \epsilon} _{Br,K}(pt,\Z/2)\simeq \oplus _{j=0}^{q-1} H^{a-b+j+(j+b)\sigma}_{Br,C_2}(pt,\Z/2)\oplus H^{a-b+j+1+(j+b)\sigma}_{Br,C_2}(pt,\Z/2)\oplus H^{a-b+q+(b+q)\sigma}_{Br,C_2}(pt,\Z/2).$$ 
 
 Notice that the realization from Proposition \ref{classic} applies to each term of the direct sum because $b,q\geq 0$. The statement of the theorem follows now in the case $b,q\geq 0$.
 
 We have that 
 $$W_n=A(n\sigma)\setminus {0}_+\rightarrow S^0\rightarrow T^{n\sigma}=\frac{A(n\sigma)}{A(n\sigma)\setminus {0}}$$
 is an equivariant  $C_2$-motivic cofiber sequence that gives a long exact sequence in Bredon motivic cohomology. 
 
  Notice that $T^{n\sigma}=S^{n\sigma}\wedge S^{n\sigma} _t=S^{2n\sigma,n\sigma}$ (\cite{HOV1}) and for $q<0$ we have by definition that $$H^{a+2q\sigma,b+q\sigma} _{C _2}(\R,\Z/2)=\rH^{a,b} _{C_2}(T^{-q\sigma},\Z/2).$$
 
 Because the realization maps $$H^{a,b} _{C _2}(W_{-q},\Z/2)\simeq H^{a-b+b\epsilon} _{Br,C_2}(S(-q\sigma\oplus -q\sigma\otimes \epsilon),\Z/2)$$ are isomorphisms when $a\leq 2b$ it implies that $$H^{a+2q\sigma,b+q\sigma} _{C _2}(\R,\Z/2)\simeq H^{a-b+q\sigma+b\epsilon+q\sigma\otimes \epsilon}_{Br,K}(pt,\Z/2)$$ is an isomorphism if $q<0$ and $b+q\geq 0$ and $a\leq 2b$ from the following diagram and 5-lemma:
  \[
\xymatrixrowsep{0.3in}
\xymatrixcolsep{0.25in}
\xymatrix@-0.9pc{
H^{a-1, b}_{C _2}(\R,\Z/2) \ar[r]\ar[d]^{\simeq} & 
H^{a-1, b}_{C_2}(W _{-q},\Z/2) \ar[r] \ar[d]^{\simeq}& 
\rH^{a, b}_{C _2}(T^{-q},\Z/2)\ar[r] \ar[d] &
H^{a,b} _{C _2}(\R,\Z/2)\ar[d]^{\simeq}\ar[r]&
H^{a,b}_{C _2}(W_{-q},\Z/2)\ar[d]^{\simeq}
\\ 
H^{a-b-1+b\sigma}_{Br,C_2}(pt)\ar[r] &
H^{a-b-1+b\sigma}_{Br,C_2}(S(-q\sigma-q\sigma\otimes \epsilon)) \ar[r] & 
\rH^{a-b+b\sigma}_{Br,C_2}(S^{-q\sigma-q\sigma\otimes \epsilon})  \ar[r]&
H^{a-b+b\sigma} _{Br,C_2}(pt)\ar[r]&
H^{a-b+b\sigma} _{Br,C_2}(S(-q\sigma-q\sigma\otimes\epsilon))
}
\]
Notice that if $a>2b$  then $H^{a+2q\sigma,b+q\sigma} _{C _2}(\R,\Z/2)=0$ if $q<0$ and $b\geq -q$. This follows from [\cite{HOV2}, Proposition 2.9] if $a\geq 2b+2$ and from Proposition \ref{v2n} in the case $a=2b+1$.
 But $$H^{a-b+q\sigma+b\epsilon+q\sigma\otimes \epsilon}_{Br,K}(pt,\Z/2)=0$$ if $a-b>b\geq -q>0$ (see Proposition \ref{cvan}) so the realization maps are isomorphisms in this range. We conclude that 
$$H^{a+2q\sigma,b+q\sigma} _{C _2}(\R,\Z/2)\simeq H^{a-b+q\sigma+b\epsilon+q\sigma\otimes \epsilon}_{Br,K}(pt,\Z/2)$$ is an isomorphism if $q<0$ and $b+q\geq 0$ which concludes the proof.
\end{proof}

The next proposition settles the realization maps in the range $b+q\geq 0$, $b\geq 0$ of Figure \ref{TEST1}.

\begin{theorem}\label{posisom} Let $b\geq 0$ and $b+q\geq 0$. Then the realization map is an isomorphism 
$$H^{a+p\sigma,b+q\sigma}_{C _2}(\R,\Z/2)\simeq H^{a-b+(p-q)\sigma+b\epsilon+q\sigma\otimes \epsilon}_{Br,K}(pt,\Z/2).$$
\end{theorem}
\begin{proof} We have the following diagram for $p+1=2q$ and arbitrary $a$:
\[
\xymatrixrowsep{0.3in}
\xymatrixcolsep{0.17in}
\xymatrix@-0.9pc{
H^{a+(p+1)\sigma, b+q\sigma}_{C _2}(\R,\Z/2) \ar[r]\ar[d]^{\simeq} & 
H^{a+p+1, b+q}(\R,\Z/2) \ar[r] \ar[d]^{\simeq} & 
H^{a+1+p\sigma, b+q\sigma}_{C _2}(\R,\Z/2)\ar[r] \ar[d] &
H^{a+1+(p+1)\sigma,b+q\sigma} _{C _2}(\R,\Z/2)\ar[d]^{\simeq}\ar[r]&
H^{a+p+2,b+q}(\R,\Z/2)\ar[d]^{\simeq}
\\ 
H^{\st}_{Br,K}(pt,\Z/2)\ar[r] &
H^{\st}_{Br,K}(pt,\Z/2) \ar[r] & 
H^{\st}_{Br,K}(pt,\Z/2)  \ar[r]&
H^{\st} _{Br,K}(pt,\Z/2)\ar[r]&
H^{\st} _{Br,K}(pt,\Z/2)
}
\]
and conclude from five lemma, Proposition \ref{2q} and Proposition \ref{classic} the isomorphism of the realization maps for $p=2q-1$ and arbitrary $a$. Downward induction concludes the theorem for $p<2q$. 

For the case $p=2q$ and arbitrary $a$ we use the diagram:
\[
\xymatrixrowsep{0.3in}
\xymatrixcolsep{0.25in}
\xymatrix@-0.9pc{
H^{a-1+p, b+q}(\R,\Z/2) \ar[r] \ar[d]^{\simeq} & 
H^{a+p\sigma, b+q\sigma}_{C _2}(\R,\Z/2)\ar[r] \ar[d]^{\simeq}&
H^{a+(p+1)\sigma,b+q\sigma} _{C _2}(\R,\Z/2)\ar[d]\ar[r]&
H^{a+p+1,b+q}(\R,\Z/2)\ar[r]\ar[d]^{\simeq}&
H^{a+1+p\sigma,b+q\sigma} _{C _2}(\R,\Z/2)\ar[d]^{\simeq} 
\\ 
H^{\st}_{Br,K}(pt,\Z/2) \ar[r] & 
H^{\st}_{Br,K}(pt,\Z/2)  \ar[r]&
H^{\st} _{Br,K}(pt,\Z/2)\ar[r]&
H^{\st} _{Br,K}(pt,\Z/2)\ar[r]&
H^{\st} _{Br,K}(pt,\Z/2).
}
\]
Using five lemma, Proposition \ref{2q} and Proposition \ref{classic} we conclude that the realization maps are isomorphic for $p=2q+1$ and arbitrary $a$. Upward induction concludes the theorem for $p>2q$.  
\end{proof}
\subsection{The case $b\geq 0$ and $b+q<0$ revisited}

Consider the case $b\geq 0$ and $b+q<0$ and we use the ideas from the previous section to discuss the realization maps in this range. We notice from the proof of Theorem \ref{2q} that if $q<0$ and $b\geq 0$ we  have a diagram
 \[
\xymatrixrowsep{0.3in}
\xymatrixcolsep{0.25in}
\xymatrix@-0.9pc{
H^{a-1, b}_{C _2}(\R,\Z/2) \ar[r]\ar[d] & 
H^{a-1, b}_{C_2}(W _{-q},\Z/2) \ar[r] \ar[d]& 
\rH^{a, b}_{C _2}(T^{-q},\Z/2)\ar[r] \ar[d] &
H^{a,b} _{C _2}(\R,\Z/2)\ar[d]\ar[r]&
H^{a,b}_{C _2}(W_{-q},\Z/2)\ar[d]
\\ 
H^{a-b-1+b\sigma}_{Br,C_2}(pt)\ar[r] &
H^{a-b-1+b\sigma}_{Br,C_2}(S(-q\sigma-q\sigma\otimes \epsilon)) \ar[r] & 
\rH^{a-b+b\sigma}_{Br,C_2}(S^{-q\sigma-q\sigma\otimes \epsilon})  \ar[r]&
H^{a-b+b\sigma} _{Br,C_2}(pt)\ar[r]&
H^{a-b+b\sigma} _{Br,C_2}(S(-q\sigma-q\sigma\otimes\epsilon)).
}
\]
It implies from 5-lemma that the realization maps $$H^{a+2q\sigma,b+q\sigma}_{C_2}(\R)\rightarrow H^{a-b+q\sigma+b\epsilon+q\sigma\otimes\epsilon}_{Br,K}(pt),$$ given by the middle vertical map above, are isomorphisms if $a\leq 2b+1$. 

If $a\geq 2b+2$ the maps are trivially injective from [\cite{HOV2}, Proposition 2.9] because the domain is zero. 

Because of the periodicity $(\sigma,0)$ in this range for Bredon motivic cohomology of a real numbers we conclude the following:
\begin{proposition} \label{ncinj2}Let $b\geq 0$ and $b+q<0$. Then the realization maps $$H^{a+p\sigma,b+q\sigma}_{C_2}(\R)\rightarrow H^{a-b+(p-q)\sigma+b\epsilon+q\sigma\otimes\epsilon}_{Br,K}(pt)$$ are monomorphisms if $p\leq 2q$.
\end{proposition}
In the case $a\geq 2b+2$ the maps in Proposition \ref{ncinj2} are monomorphisms because by the periodicity the domain is zero. Proposition \ref{ncinj2} is equivalent with part of Proposition \ref{ncinj} in the view of Corollary \ref{itilde}. The above realization maps are monomorphisms in the case $p>2q$, $a\leq 2b+1$ from Proposition \ref{ncinj}. In particular, it shows that the realization maps in the region $b+q<0$, $b\geq 1$ of Figure \ref{TEST1} are all monomorphisms.

\begin{remark}\label{ncc} Some of the realization maps in the region $b+q<0$, $b\geq 1$ of Figure \ref{TEST1} are monomorphisms, but not isomorphisms. For example, in the case $a=2b+1$, according to Proposition \ref{ncinj}, we have the following short exact sequence:
$$0\rightarrow \rH^{2+\epsilon-3\sigma\otimes\epsilon}_{Br,K}(\tilde{E}_{C_2}C_2)=\Z/2\rightarrow H^{2+\epsilon-3\sigma\otimes\epsilon}_{Br,K}(pt)=\Z/2\oplus\Z/2\rightarrow H^{2+\epsilon-3\sigma\otimes\epsilon}_{Br,K}(E_{C_2}C_2)=\Z/2\rightarrow 0,$$
thus the realization map in motivic bidegree $(3-3\sigma,1-3\sigma)$ is a nontrivial monomorphism which is not surjective. 

If we compute the realization map in motivic bidegree $(3-2\sigma,1-2\sigma)$ we see that it is an isomorphism. Indeed, from Proposition \ref{ncinj}, we have the following short exact sequence:
$$0\rightarrow \rH^{2+\epsilon-2\sigma\otimes\epsilon}_{Br,K}(\tilde{E}_{C_2}C_2)=\Z/2\rightarrow H^{2+\epsilon-2\sigma\otimes\epsilon}_{Br,K}(pt)=\Z/2\rightarrow H^{2+\epsilon-2\sigma\otimes\epsilon}_{Br,K}(E_{C_2}C_2)=0\rightarrow 0.$$

\end{remark}

\section{Bredon motivic cohomology of $\EG C_2$}
In this section we completely compute the Bredon motivic cohomology groups and ring of $\EG C_2$ over the real numbers. This is Theorem \ref{Borel} below. The methods we use here will also reprove in an easier way the computation of Borel motivic cohomology ring of the complex numbers given in \cite{HOV2}. As an application, we prove in this section that all the realization maps for Bredon motivic cohomology of real numbers in the range $b<0$, $b+q\geq 0$ are monomorphisms (Corollary \ref{EG} below). According to Remark \ref{tcc} below, they are not necessary isomorphisms.
\begin{proposition} \label{2qEG} The realization map
$$H^{a+2q\sigma,b+q\sigma}_{C _2}(\EG C _2,\Z/2)\rightarrow H^{a-b+q\sigma+b\epsilon+q\sigma\otimes \epsilon}_{Br,K}(E_{\Sigma _2}C _2,\Z/2),$$ is an isomorphism for any $a\leq 2b+2$, $b+q\geq 0$. For $b+q<0$ or $a\geq 2b+1$ the realization map is zero because the domain is zero.  
\end{proposition}
\begin{proof}
 According to the periodicity of the Bredon  motivic cohomology of $\EG C _2$ we have that 
 $$H^{a+2q\sigma,b+q\sigma}_{C _2}(\EG C _2,\Z/2)\simeq H^{a+2q,b+q}_{C _2}(\EG C _2,\Z/2)\simeq H^{a+2q,b+q} (\BG C _2,\Z/2).$$
 Because of the vanishing of motivic cohomology (see \cite{MVW}) we have that the above isomorphisms are zero if $b+q<0$ or if $a\geq 2b+1$. Also, because of Lemma \ref{gap}, periodicity and the isomorphisms below, the codomains are also zero when $a=2b+1$ or $a=2b+2$. 
 
 Looking to the realization maps for $b+q\geq 0$ and using Corollary \ref{cr}, we have the following diagram
\[
\xymatrixrowsep{0.4in}
\xymatrixcolsep{0.35in}
\xymatrix@-0.9pc{
H^{a+2q\sigma,b+q\sigma}_{C _2}(\EG C _2,\Z/2)\ar[r]^{\stackrel{\kappa_2^q}{\simeq}}\ar[d]&
H^{a+2q,b+q}_{C _2}(\EG C _2,\Z/2)\ar[d]^{\simeq}\ar[r]^{\simeq}&
H^{a+2q,b+q}(\BG C _2,\Z/2)\ar[d]^{\simeq}
\\ 
H^{a-b+q\sigma+b\epsilon+q\sigma\otimes \epsilon} _{Br,K}(E_{\Sigma_2}{C _2},\Z/2)\ar[r]^{\stackrel{\kappa_2^q}{\simeq}}&
H^{a-b+q+(b+q)\epsilon} _{Br,C_2}(E_{\Sigma_2}{C _2},\Z/2)\ar[r]^{\simeq}&
H^{a-b+q+(b+q)\epsilon} _{Br,C_2}(B _{\Sigma _2}{C _2},\Z/2).
}
\]
The right vertical map is an isomorphism if $a\leq 2b+2$ from Corollary \ref{cr} and Corollary \ref{comp3}. We notice that in the case $a=2b+1$ or $a=2b+2$ both cohomologies of the right vertical map are zero. The upper horizontal maps and the left horizontal map are isomorphisms from periodicity and properties of Borel motivic cohomology (\cite{HOV1} Proposition 3.16). The right lower horizontal map is an isomorphism because $E_{\Sigma_2}{C_2}$ is a $C_2\times \Sigma_2$ topological space and therefore $$H^{a-b+q+(b+q)\epsilon} _{Br,C_2}(E_{\Sigma_2}{C _2},\Z/2)\simeq H^{a-b+q+(b+q)\epsilon} _{Br,C_2}(E_{\Sigma_2}{C _2}/C_2,\Z/2)=H^{a-b+q+(b+q)\epsilon} _{Br,C_2}(B_{\Sigma_2}{C _2},\Z/2).$$
It implies that the left vertical map is an isomorphism for $a\leq 2b+2$ and $b+q\geq 0$.
\end{proof}
\begin{proposition} Let $b+q\geq 0$. The realization map induces an isomorphism 
$$H^{a+p\sigma,b+q\sigma}_{C _2}(\EG C _2,\Z/2)\simeq H^{a-b+(p-q)\sigma+b\epsilon+q\sigma\otimes \epsilon}_{Br,K}(E_{\Sigma _2}C _2,\Z/2)$$ for any $a\leq 2b+2$. 
\end{proposition}
\begin{proof}
 We have the following diagram:
\[
\xymatrixrowsep{0.3in}
\xymatrixcolsep{0.25in}
\xymatrix@-0.9pc{
H^{a-1+2q\sigma, b+q\sigma}_{C _2}(\EG C _2) \ar[r]\ar[d]^{\simeq} & 
H^{a-1+2q, b+q}(\R)\ar[r] \ar[d]^{\simeq} & 
H^{a+(2q-1)\sigma, b+q\sigma}_{C _2}(\EG C _2)\ar[r] \ar[d] &
H^{a+2q\sigma,b+q\sigma} _{C _2}(\EG C _2)\ar[d]^{\simeq}\ar[r]&
H^{a+2q,b+q}(\R)\ar[d]^{\simeq}
\\ 
H^{\st}_{Br,K}(E _{\Sigma_2}C _2)\ar[r] &
H^{\st}_{Br,K}(pt) \ar[r] & 
H^{\st}_{Br,K}(E_{\Sigma_2}C_2)  \ar[r]&
H^{\st} _{Br,K}(E _{\Sigma_2}C _2)\ar[r]&
H^{\st} _{Br,K}(pt)
}
\]
The upper sequence is induced by the motivic $C _2$-cofiber sequence $$\EG C _{2+}\wedge C _{2+}\simeq C _{2+}\rightarrow \EG C _{2+}\rightarrow \EG C _{2+}\wedge S^\sigma.$$ The lower sequence is induced by the $C _2\times \Sigma_2$-equivariant cofiber sequence induced by the above cofiber sequence through realization
$$E _{\Sigma_2}C _{2+}\wedge C _{2+}\simeq C _{2+}\rightarrow E_{\Sigma_2}C _{2+}\rightarrow 
E_{\Sigma_2}C _{2+}\wedge S^\sigma.$$
Here we used that $\EG C _2$ is non-equivariantly contractible and that the isomorphism of the left term of the motivic cofiber sequence commutes with the realization (see \cite{HOV1}).
 
The first and fourth horizontal maps are isomorphisms from Proposition \ref{2qEG} because $b+q\geq 0$. The second and fourth maps are isomorphisms from Proposition \ref{classic}. Therefore the middle map of the diagram is an isomorphism for  $a\leq 2b+2$ by  five lemma. Now downward induction concludes that the cycle map from the statement of the proposition is an isomorphism for $p\leq 2q$, $a\leq 2b+2$. Upward induction in the diagram below concludes the case $p>2q$, $a\leq 2b+2$.
\[
\xymatrixrowsep{0.3in}
\xymatrixcolsep{0.25in}
\xymatrix@-0.9pc{
H^{a+2q, b+q}_{C_2}(\R) \ar[r]\ar[d]^{\simeq} & 
H^{a+2q\sigma, b+q\sigma}_{C_2}(\EG C_2)\ar[r] \ar[d]^{\simeq} & 
H^{a+(2q+1)\sigma, b+q\sigma}_{C _2}(\EG C _2)\ar[r] \ar[d] &
H^{a+1+2q,b+q} _{C _2}(\R)\ar[d]^{\simeq}\ar[r]&
H^{a+1+2q\sigma,b+q\sigma}_{C_2}(\EG C_2)\ar[d]^{\simeq}
\\ 
H^{\st}_{Br,K}(pt)\ar[r] &
H^{\st}_{Br,K}(E _{\Sigma_2}C _2) \ar[r] & 
H^{\st}_{Br,K}(E_{\Sigma_2}C_2)  \ar[r]&
H^{\st} _{Br,K}(pt)\ar[r]&
H^{\st} _{Br,K}(E _{\Sigma_2}C _2)
}
\]

\end{proof}
The following theorem computes all of the Borel motivic cohomology groups of the real numbers. 
\begin{theorem} \label{Borel} $H^{a+p\sigma,b+q\sigma} _{C _2}(\EG C _2,\Z/2)\simeq H^{a-2b+(p-q+b)\sigma+(b+q)\sigma\otimes\epsilon} _{Br,K}(pt,\Z/2)$ for any $b+q\geq 0$. For $b+q<0$ we have $H^{a+p\sigma,b+q\sigma} _{C _2}(\EG C _2,\Z/2)=0$.
\end{theorem}
\begin{proof} The proof follows from the following diagram for $b+q\geq 0$. The upper box in Figure \ref{TEST4} is commutative because the  realization is a ring map, and the second box is commutative because it is given by the realization applied to the isotropy sequences. The top left vertical and top right vertical isomorphisms follow from the corresponding periodicities. The bottom horizontal map follows from the first part ($b,q\geq 0$ case) of the proof of Proposition \ref{posisom}. The left lower vertical map is an isomorphism from the motivic isotropy sequence together with [\cite{V}, Proposition 4.3] which says that $\rH^{a+p\sigma,0}(\wt\EG C _2)=0$ for any $a,p\in \Z,$ and the from the  $(0,\sigma)$ periodicity of the Bredon motivic cohomology $\wt\EG C _2$.
\begin{equation}
\label{TEST4}
\begin{tikzcd}
	{H^{a+p\sigma,b+q\sigma}_{C _2}(\EG C _2,\Z/2)} && {H^{a-b+(p-q)\sigma+b\epsilon+q\sigma\otimes\epsilon} _{Br,K}(E_{C _2}C _2,\Z/2)} \\
	{H^{a-2b+(p+2b)\sigma,(b+q)\sigma}_{C _2}(\EG C _2,\Z/2)} && {H^{a-2b+(p-q+b)\sigma+(b+q)\sigma\otimes\epsilon}_{Br,K}(E _{C _2}C _2,\Z/2)} \\
	{H^{a-2b+(p+2b)\sigma,(b+q)\sigma}_{C _2}(\R,\Z/2)} && {H^{a-2b+(p-q+b)\sigma+(b+q)\sigma\otimes\epsilon} _{Br,K}(pt,\Z/2)}
	\arrow[from=1-1, to=1-3]
	\arrow["{\simeq \kappa_2^{b}}"', from=1-1, to=2-1]
	\arrow["\simeq", from=3-1, to=2-1]
	\arrow["\simeq"', from=3-1, to=3-3]
	\arrow[from=3-3, to=2-3]
	\arrow[from=2-1, to=2-3]
	\arrow["{\kappa_2^{b}\simeq }", from=1-3, to=2-3]
\end{tikzcd}
\end{equation}
If $b+q<0$ then by periodicity we have $H^{a+p\sigma,b+q\sigma} _{C _2}(\EG C _2,\Z/2)\simeq H^{a+2q+(p-2q)\sigma,b+q}_{C_2}(\EG C_2,\Z/2)$. Then the last term is zero in the $b+q<0$ range from the proof of Proposition \ref{tilde}.
\end{proof}
\begin{corollary} \label{EG} Let $b+q\geq 0$. The realization map induces an isomorphism 
$$H^{a+p\sigma,b+q\sigma}_{C _2}(\EG C _2,\Z/2)\simeq H^{a-b+(p-q)\sigma+b\epsilon+q\sigma\otimes \epsilon}_{Br,K}(E_{\Sigma _2}C _2,\Z/2)$$ for any $a\leq 2b+2$ and a monomorphism (but not surjective) if $a\geq 2b +3$. 

Moreover if $b<0$ and $a\leq 2b+2$ there is an identification of the realization maps (in addition to the identification of cohomology groups) i.e.
\begin{equation}
\label{TEST8}
\begin{tikzcd}
	{H^{a+p\sigma, b+q\sigma}_{C _2}(\R,\Z/2)} && {H^{a+p\sigma, b+q\sigma} _{C _2}(\EG C _2,\Z/2)} \\
	\\
	{H^{a-b+(p-q)\sigma+b\epsilon+q\sigma\otimes\epsilon}_{Br,K}(pt,\Z/2)} && {H^{a-b+(p-q)\sigma+b\epsilon+q\sigma\otimes\epsilon}_{Br,K}(E _{\Sigma_2}C _2,\Z/2)}
	\arrow["\simeq", from=1-1, to=1-3]
	\arrow[from=1-1, to=3-1]
	\arrow["\simeq"', from=3-1, to=3-3]
	\arrow["\simeq", from=1-3, to=3-3]
\end{tikzcd}
\end{equation}
The horizontal maps are induced by the isotropy sequences. Notice that if $b<0, b+q\geq 0$ the Diagram \ref{TEST8} shows that the realization maps
$$H^{a+p\sigma, b+q\sigma}_{C _2}(\R,\Z/2)\hookrightarrow H^{a-b+(p-q)\sigma+b\epsilon+q\sigma\otimes\epsilon}_{Br,K}(pt,\Z/2)$$
are isomorphisms for $a\leq 2b+2$ and monomorphisms for $a\geq 2b+3$. 
\end{corollary}
\begin{proof}
If $b<0$, from the motivic isotropy sequence \ref{eqn:cof2}, we have that
$$H^{a+p\sigma, b+q\sigma}_{C _2}(\R,\Z/2)\simeq H^{a+p\sigma, b+q\sigma} _{C _2}(\EG C _2,\Z/2)$$
are isomorphisms.

The lower horizontal map in the Diagram \ref{TEST8} is an isomorphism for $a\leq 2b+2$ and $b<0$ from the $C_2\times \Sigma_2$ topological isotropy sequence and Corollary \ref{comp3}.

We can see from Diagram \ref{TEST4}  that the lower right vertical map  is an isomorphism if $a\leq 2b+2$ because $\rH^{a}_{Br,K}(\tilde{E}_{\Sigma}C_2,\Z/2)=0$ if $a\leq 3$ (Corollary \ref{comp3}). The monomorphism of the realization map from the statement of the corollary, when $a\geq 2b+3$, follows from Theorem \ref{isov} ($b+q\geq 0$) and Diagram \ref{TEST4}. 

It follows that the realization maps for the Bredon motivic cohomology of $\EG C_2$ are isomorphisms if $a\leq 2b+2$ and monomorphisms for $a\geq 2b+3$. We conclude from Diagram \ref{TEST8} that the realization maps
$$H^{a+p\sigma, b+q\sigma}_{C _2}(\R,\Z/2)\rightarrow H^{a-b+(p-q)\sigma+b\epsilon+q\sigma\otimes\epsilon}_{Br,K}(pt,\Z/2)$$
are isomorphisms if $a\leq 2b+2$ and $b<0$. When $a>2b+3$ and $b<0$ these maps are monomorphisms because in Diagram \ref{TEST8} the right vertical maps are monomorphisms and the upper horizontal maps are isomorphisms.

From the $C_2\times \Sigma_2$ topological isotropy sequence we have
$$\rH^{a-2b}_{Br,C_2}(\tilde{E}_{\Sigma _2}C _2)\rightarrow H^{a-2b+(p-q+b)\sigma+(b+q)\sigma\otimes\epsilon}_{Br,K}(pt) \rightarrow H^{a-2b+(p-q+b)\sigma+(b+q)\sigma\otimes\epsilon}_{Br,K}(E_{\Sigma _2}C _2)\rightarrow \rH^{a-2b+1}_{Br,C_2}(\tilde{E}_{\Sigma_2}C _2).$$
Moreover $\rH^{n}_{Br,C_2}(\tilde{E}_{\Sigma_2}C _2)=0$ if $n\leq 3$ and all the elements for $n>0$ from $$\rH^{n+1}_{Br,C_2}(\tilde{E}_{\Sigma _2}C _2)\simeq \rH^{n}_{Br,C_2}(B_{\Sigma _2}C _2)$$ are either zero or of the form $\frac{\theta_2}{x_2^{n'}y_2^m}c^pb^q$ with 4-degree $(2+m+q,0,-2-n'-m+p+q,0)$ where $p+q=2+m+n'$ and $2+m+q\geq 2$ and $p=0,1$ (see proof of Lemma \ref{gap}). One can notice that $\Sigma\theta_2 cb\in \rH^{4}_{Br,C_2}(\tilde{E}_{\Sigma_2}C _2,\Z/2)\simeq \Z/2$ is the nonzero element.  Moreover if $n\geq 5$ then $\rH^{n}_{Br,C_2}(\tilde{E}_{\Sigma_2}C_2,\Z/2)\neq 0$  because we can always choose $m,n',p,q\geq 0$ such that $m+q=n-2$ and $p=2+n'+m-q$ and $p=0,1$. For example, we can take the nonzero elements $\frac{\theta_2}{x_2^{p+n-4}}c^pb^{n-2}\in \rH^{n}_{Br,C_2}(B_{\Sigma_2}C_2,\Z/2)$ with $n\geq 4$ and $p=0,1$. They have their nonzero images $\Sigma \frac{\theta_2}{x_2^{p+n-4}}c^pb^{n-2}\in  \rH^{n+1}_{Br,C_2}(\tilde{E}_{\Sigma_2}C_2,\Z/2)\neq 0$.

Because the map
$$\rH^{n}_{Br,C_2}(\tilde{E}_{\Sigma _2}C _2)\rightarrow H^{n+(p-q+b)\sigma+(b+q)\sigma\otimes\epsilon}_{Br,K}(pt)$$
is the zero map from Theorem \ref{isov} ($b+q\geq 0$) for any $n\in \Z$ we obtain a split short exact sequence of $\Z/2$- vector spaces
$$0\rightarrow H^{a-2b+(p-q+b)\sigma+(b+q)\sigma\otimes\epsilon}_{Br,K}(pt) \stackrel{h}{\rightarrow} H^{a-2b+(p-q+b)\sigma+(b+q)\sigma\otimes\epsilon}_{Br,K}(E_{\Sigma _2}C _2)\rightarrow \rH^{a-2b+1}_{Br,C_2}(\tilde{E}_{\Sigma _2}C _2)\rightarrow 0.$$
From Diagram \ref{TEST4} we see that the surjectivity of the realization maps of $\EG C_2$ is equivalent to the surjectivity of the maps $h$ in the above short exact sequence. Therefore we deduce that the realization maps of $\EG C_2$ are not surjective if $a\geq 2b+3$.
 \end{proof}
 Corollary \ref{EG} describes the realization maps in the region $b+q\geq 0$, $q\geq 0$, $b<0$ of Figure \ref{TEST1}. Theorem \ref{Borel} computes all of the groups in the region $b+q\geq 0$, $b<0$, $q\geq 0$ of Figure \ref{TEST1}.
 
 The next remark gives examples of the realization maps in the range $b<0$, $b+q\geq 0$ and $a\geq 2b+3$ that are monomorphisms, but not isomorphisms.
 \begin{remark} \label{tcc}The realization homomorphisms $$H^{a+p\sigma, b+q\sigma}_{C _2}(\R,\Z/2)\hookrightarrow H^{a-b+(p-q)\sigma+b\epsilon+q\sigma\otimes\epsilon}_{Br,K}(pt,\Z/2)$$
are monomorphisms for $a\geq 2b+3$, $b<0$, $b+q\geq 0$ from Theorem \ref{EG}. There are examples where these realizations are not necessary isomorphisms. For example $$H^{-1,-2+2\sigma}_{C_2}(\R)\simeq \Z/2\hookrightarrow H^{1-2\sigma-2\epsilon+2\sigma\otimes\epsilon}_{Br,K}(pt)\simeq \Z/2\oplus\Z/2,$$
as well as 
$$H^{0,-2+2\sigma}_{C_2}(\R)\simeq \Z/2\hookrightarrow H^{2-2\sigma-2\epsilon+2\sigma\otimes\epsilon}_{Br,K}(pt)\simeq \Z/2\oplus\Z/2,$$
with the computations following from Theorem \ref{Borel} and Proposition \ref{caseMC2}. On the other hand
$$H^{1,-2+2\sigma}_{C_2}(\R)=0\simeq H^{3-2\sigma-2\epsilon+2\sigma\otimes\epsilon}_{Br,K}(pt)=0,$$
so we can also have isomorphisms in this range. For non-trivial isomorphisms, we have for example
$$H^{1,-2+3\sigma}_{C_2}(\R)=\Z/2\simeq H^{3-3\sigma-2\epsilon+3\sigma\otimes\epsilon}_{Br,K}(pt)=\Z/2.$$
 \end{remark}

We start analyzing the ring structure of Borel motivic cohomology of real numbers. First, we decide which topological cohomological classes come from corresponding algebraic cohomological classes via the realization maps. In these cases, we denote the algebraic cohomology class with the same notation with the induced topological cohomology class.

We have that $$x _3\in H^{\sigma,\sigma}_{C_2}(\R,\Z/2)\simeq H^{\sigma\otimes \epsilon} _{Br,C_2}(pt,\Z/2)\simeq H^{\sigma,\sigma}_{C_2}(\EG C _2,\Z/2)\simeq \Z/2$$ is the unique nontrivial class. Also $$y _3\in H^{\sigma-1,\sigma}_{C_2}(\R,\Z/2)\simeq H^{-1+\sigma\otimes\epsilon} _{Br,C_2}(pt,\Z/2)\simeq H^{\sigma-1,\sigma}_{C_2}(\EG C_2,\Z/2)\simeq \Z/2$$ is the unique nontrivial class. It is also obvious that $$y _1\in H^{\sigma-1,0}_{C_2}(\R,\Z/2)\simeq H^{-1+\sigma} _{Br,C_2}(pt,\Z/2)\simeq H^{\sigma-1,0}_{C_2}(\EG C_2,\Z/2)\simeq \Z/2$$ and $$x _1\in H^{\sigma,0}_{C_2}(\R,\Z/2)\simeq H^{\sigma} _{Br,C_2}(pt,\Z/2)\simeq H^{\sigma,0}_{C_2}(\EG C_2,\Z/2)\simeq \Z/2$$ are the unique nontrivial classes. 

 The cohomology groups of $E_{\Sigma_2}C_2$ in the above indexes are also isomorphic through the realization because they fulfill the condition $a\leq 2b+2$ (see Corollary \ref{EG}). 
 
 For example we have 
\[\begin{tikzcd}
	{x_3\in H^{\sigma,\sigma}_{C_2}(\EG C _2,\Z/2)} & {H^{\sigma,\sigma}_{C_2}(\R,\Z/2)} \\
	{H^{\sigma\otimes\epsilon}_{Br,C_2}(E_{\Sigma}C _2,\Z/2)} & {H^{\sigma\otimes\epsilon}_{Br,C_2}(pt,\Z/2)\simeq \Z/2}
	\arrow["\simeq", from=1-2, to=1-1]
	\arrow["\simeq"', from=1-2, to=2-2]
	\arrow["\simeq", from=1-1, to=2-1]
	\arrow["\simeq"', from=2-2, to=2-1]
\end{tikzcd}\]

We also have the following diagram:
\[\begin{tikzcd}
	{\kappa_2\in H^{2\sigma-2,\sigma-1}_{C_2}(\EG C _2,\Z/2)} & {H^{2\sigma-2,\sigma-1}_{C_2}(\R,\Z/2)} \\
	{H^{-1+\sigma-\epsilon+\sigma\otimes\epsilon}_{Br,K}(E_{\Sigma}C _2,\Z/2)} & {H^{-1+\sigma-\epsilon+\sigma\otimes\epsilon}_{Br,K}(pt,\Z/2)\simeq \Z/2}
	\arrow["\simeq", from=1-2, to=1-1]
	\arrow["\simeq"', from=1-2, to=2-2]
	\arrow["\simeq", from=1-1, to=2-1]
	\arrow["\simeq"', from=2-2, to=2-1]
\end{tikzcd}\]
The top left corner is generated by the class $\kappa_2$ and the bottom right corner is generated by its image, which we also denote by $\ka_2.$ In conclusion $\kappa_2$ becomes invertible in the cohomology ring $H^{\st}_{Br,K}(\ek,\Z/2)$ because the realization is a ring map. We have $$\theta_1\kappa _2=\io_3\in  H^{0,\sigma-1}_{C_2}(\R,\Z/2)\simeq H^{1-\sigma-\epsilon+\sigma\otimes\epsilon}_{Br,K}(pt,\Z/2)\simeq \Z/2.$$  The product is nontrivial because its realization is nontrivial.

We also have 
\[\begin{tikzcd}
	{\theta_1\in H^{2-2\sigma,0}_{C_2}(\EG C _2,\Z/2)} & {H^{2-2\sigma,0}_{C_2}(\R,\Z/2)} \\
	{H^{2-2\sigma}_{Br,C_2}(E_{\Sigma _2}C _2,\Z/2)} & {H^{2-2\sigma}_{Br,C_2}(pt,\Z/2)\simeq \Z/2}
	\arrow["\simeq", from=1-2, to=1-1]
	\arrow["\simeq"', from=2-2, to=1-2]
	\arrow["\simeq", from=1-1, to=2-1]
	\arrow["\simeq"', from=2-2, to=2-1]
\end{tikzcd}\]

The Borel motivic cohomology of real numbers is computed in the following theorem as a subring of the $RO(C_2\times \Sigma_2)$-graded Bredon cohomology of $E_{\Sigma_2}C_2$.
\begin{theorem} \label{compEG} We have the following commutative diagram of rings 
\[\begin{tikzcd}
	{H^{\star,\star}_{C_2}(\EG C _2,\Z/2)} && {H^{\star}_{Br,C_2}(E_{\Sigma_2} C _2,\Z/2)} \\
	{H^{\star,0}_{C_2}(\R,\Z/2)[x _3,y _3,\kappa_2^{\pm1}]} && {\MMt[x _3,y _3,\kappa_2^{\pm1}]}
	\arrow[tail, from=1-1, to=1-3]
	\arrow["\simeq", from=2-1, to=1-1]
	\arrow["\simeq"', from=2-1, to=2-3]
	\arrow[tail, from=2-3, to=1-3]
\end{tikzcd}\]
Thus, $$H^{\st,\st}_{C_2}(\EG C _2,\Z/2)\simeq \MMt[x _3,y _3,\kappa_2^{\pm 1}].$$

Moreover if $p\leq q-b-2$ the Borel motivic cohomology of real numbers has all nonzero elements nilpotent. 
\end{theorem}
\begin{proof} Firstly, we have that $H^{\st,0}_{C_2}(\R,\Z/2)\simeq H^{\st,0}_{C_2}(\EG C _2,\Z/2)$. Secondly, according to Proposition \ref{Borel}, the Borel motivic cohomology ring of real numbers, after multiplication with $\kappa_2^{\pm{1}}$, can be reduced to a ring that is isomorphic (as rings) to the cohomology ring $\oplus_{a,p\in\Z,q\geq 0}H^{a+p\sigma+q\sigma\otimes\epsilon} _{Br,K}(pt,\Z/2)$. Therefore we have an isomorphism  of rings 
$$H^{\st,\st}_{C_2}(\EG C _2,\Z/2)\simeq (\oplus_{a,p\in\Z,q\geq 0}H^{a+p\sigma+q\sigma\otimes\epsilon} _{Br,K}(pt,\Z/2))[\ka_2^{\pm 1}]$$

We want to determine the ring $\oplus_{a,p\in\Z,q\geq 0}H^{a+p\sigma+q\sigma\otimes\epsilon} _{Br,K}(pt,\Z/2)$ which is a proper subring in $RO(C_2\times \Sigma_2)$-graded Bredon cohomology of a point. We distinguish two cases: 

Case 1: $p\geq 0,$ and  

Case 2: $p<0$. 

{\bf{Case 1:}} Let $p, q\geq 0$. Then the groups $H^{a+p\sigma+q\sigma\otimes\epsilon} _{Br,K}(pt,\Z/2)$ are computed by the Poincare series
$(1+x+...+x^p)(1+x+...+x^q)$ (see Theorem \ref{PC}). It implies that the groups are zero if $a>0$. We will prove that all their generators are given by elements of the form $x^{n_1} _1y^{m_1}_1x^{n_2}_3y^{m_2} _3$ with $n_1, m _1, n_2, m_2\geq 0$. Since $|x_1|=(0,1,0,0)$, $|x _3|=(0,0,0,1)$, $|y_1|=(-1,1,0,0)$ and $|y_3|=(-1,0,0,1),$ the monomial $x^{n_1} _1y^{m_1}_1x^{n_2}_3y^{m_2} _3$ has degree $(-m_1-m_2,n_1+m _1,0,n_2+m_2)$, and belongs to $H^{a+p\sigma+q\sigma\otimes\epsilon} _{Br,K}(pt,\Z/2)$.  We have that $0\leq m _1\leq p$ and $0\leq m _2\leq q$. Notice that if we fix a pair $(m_1, m_2),$ we have a unique element in the Poincare sum containing $y_1^{m _1}y_3^{m_2}$ such that $a=-m _1-m _2$. Furthermore, a term $y_1^{m _1}y_3^{m_2}$ in the Poincare sum gives a unique pair $(m _1,m_2)$ (order sensitive) in $H^{a+p\sigma+q\sigma\otimes\epsilon} _{Br,K}(pt,\Z/2)$ with $n_1=p-m_1$ and $n _2=q-m _2$. 

{\bf{Case 2:}} Let $q\geq 0$ and $p<0$. Then the groups are computed by the Poincare series $(x^{p}+...+x^{-2})(1+x+...+x^q)$ (see Theorem \ref{PC}). In particular $p\leq -2$.

We will prove that all the non-zero elements in this case are given by the generators
$\frac{\theta _1}{x^{n_1} _1y^{m_1}_1}x^{n_2} _3y^{m_2} _3$. Here $\theta _1$ has degree $(2,-2,0,0)$. A monomial $\frac{\theta _1}{x^{n_1} _1y^{m_1}_1}x^{n_2} _3y^{m_2} _3$ has degree $(2+m _1-m _2, -2-n_1-m_1,0, n_2+m_2)$. It implies that $0\leq m _2\leq q$ and $0\leq m_1\leq -p-2$. Any choice of a pair $(m _1,m _2)$ in these intervals gives a unique monomial $y_1^{-m_1-2}y_3^{m_2}$ and a unique $\frac{\theta _1}{x^{n_1} _1y^{m_1}_1}x^{n_2} _3y^{m_2} _3$ for a fixed choice of $p\leq -2,q\geq 0$. For $p=-1$ the groups are zero. The group is nonzero if $2-q\leq a\leq -p$. 

There are no relations among the generators of Case 1 and Case 2 as proved in \cite{BH}.

It implies that we have an isomorphism of  $\Z/2-$vector spaces  $$\oplus_{a,p\in\Z,q\geq 0}H^{a+p\sigma+q\sigma\otimes\epsilon} _{Br,K}(pt,\Z/2)=<\{x^{n_1} _1y^{m_1}_1x^{n_2}_3y^{m_2} _3\}_{n_1, m _1, n_2, m_2\geq 0}, \{\frac{\theta _1}{x^{n_1} _1y^{m_1}_1}x^{n_2} _3y^{m_2} _3\}_{n_1,m_1,n_2,m_2\geq 0}>.$$
Because we have the isomorphism of $\Z[x_1,y_1]-$algebras $$H^{*+*\sigma}_{Br,C_2}(pt,\Z/2)= \Z/2[x_1,y_1]\oplus \Z/2\{\frac{\theta_1}{x_1^{n_1}y_1^{m_1}}\}$$ we conclude that we have an isomorphism of rings
$$\oplus_{a,p\in\Z,q\geq 0}H^{a+p\sigma+q\sigma\otimes\epsilon} _{Br,K}(pt,\Z/2)\simeq H^{*+*\sigma}_{Br,C_2}(pt,\Z/2)[x_3,y_3].$$
The Case 2 above corresponds to the case $p\leq q-b-2$ in Diagram \ref{TEST4} concluding the last statement of the theorem.
\end{proof}

The following corollary describes the $\MMt$-module structure of the region $b+q\geq 0,b<0,q\geq 0$ in Figure \ref{TEST1}.
\begin{corollary}\label{aplic} We have the following $\MMt$-module isomorphism $$\oplus _{b<0} H^{a+p\sigma,b+q\sigma} _{C_2}(\EG C _2,\Z/2)\simeq \kappa_2(\MMt[x _3,y _3,\kappa_2]).$$
\end{corollary}
 \begin{theorem} \label{cn} We have the following diagram over the complex numbers 
\begin{equation}
\label{TEST5}
\begin{tikzcd}
	{H^{a+p\sigma,b+q\sigma}_{C_2}(\EG C _2,\Z/2)} & {H^{a+p\sigma}_{Br,C_2}(EC _2,\Z/2)} \\
	{H^{a-2b+(p+2b)\sigma,(b+q)\sigma}_{C_2}(\EG C_2,\Z/2)} & {H^{a-2b+(p+2b)\sigma}_{Br,C_2}(EC _2,\Z/2)} \\
	{H^{a-2b+(p+2b)\sigma,(b+q)\sigma}_{C_2}(\C,\Z/2)} & {H^{a-2b+(p+2b)\sigma}_{Br,C_2}(pt,\Z/2)}
	\arrow[from=1-1, to=1-2]
	\arrow["{\simeq u^b}", from=1-1, to=2-1]
	\arrow["\simeq"', from=3-1, to=2-1]
	\arrow["\simeq"', from=3-1, to=3-2]
	\arrow[from=3-1, to=3-2]
	\arrow[from=3-2, to=2-2]
	\arrow["{\simeq u^{b}}"', from=1-2, to=2-2]
	\arrow[from=2-1, to=2-2]
\end{tikzcd}
\end{equation}

Here $u$ is the cohomology class in degree $(2\sigma-2,\sigma-1)$ that gives the periodicity in Borel motivic cohomology of the complex numbers.
\end{theorem}
\begin{proof} The arguments for Diagram \ref{TEST5} are similar to those for Diagram \ref{TEST4}. The cohomology class $u$ gives the periodicity in the Bredon motivic cohomology of $\EG C_2$ and its image through the realization map  induces the periodicity in the Bredon cohomology of $EC_2$.
\end{proof}
We see from Diagram \ref{TEST5} that over the complex numbers the multiplication by $\tau_\sigma\in H^{0,\sigma}_{C_2}(\EG C _2,\Z/2)$ (the only non-trivial element in this group) gives an isomorphism on Borel motivic cohomology of real numbers if $b+q\geq 0$. This is because it induces multiplication by $1$ for the $RO(C_2)$ graded Bredon cohomology of a point in the bottom right corner. Thus, because of this isomorphism and because of the periodicity, we compute the Borel motivic cohomology ring over the complex numbers to be $$H^{\st,\st}_{C_2}(\EG C _2,\Z/2)=\MMt [\tau_\sigma, u^{\pm 1}].$$


\begin{remark} \label{Vcomp} We have the following ring structure on the motivic cohomology of $\mathbf{B}C _2$ (over the reals, \cite{Voc}, reviewed in Theorem \ref{comp2}), $$H^{*,*} (\mathbf{B}C _2,\Z/2)=H^{*,*}(\R,\Z/2)[s,t]/(s^2=\tau t+\rho s).$$ According to Diagram \ref{TEST4}, we can identify $\tau=\frac{y _1y _3}{\kappa_2}$, $\rho=\frac{x_1y_3+y_1x _3}{\kappa_2}$ and $t=\frac{x _1x _3}{\kappa_2}$. We notice that $\tau$ and $t$ are the unique generators of their groups; for $s,\rho$ we have 3 distinct choices of elements in $$H^{-1+\sigma+\sigma\otimes\epsilon}_{Br,K}(pt,\Z/2)=\Z/2(y_1x _3)\oplus \Z/2 (x _1y _3),$$ and according to the relation $s^2=\tau t+\rho s,$ the choice of $\rho$ is unique. Because the realization of $\rho$ is $x_2$ and the realization of $\tau$ is $y_2$ from Proposition \ref{classic} we obtain the following relations 
in  $RO(C_2\times \Sigma_2)$-graded Bredon cohomology of $E_{\Sigma_2}C_2$ 
$$\kappa_2y_2=y_1y_3,$$
$$\kappa_2x_2=x_1y_3+y_1x_3.$$
Because of Theorem \ref{isov} we can view these relations in $RO(C_2\times \Sigma_2)$-graded Bredon cohomology of a point giving in this way an alternate proof of these relations to the proof in \cite{BH}.

There are two equally good choices for $s$: $s=\frac{y_1x_3}{\kappa_2}$ and $s=\frac{y_3x_1}{\kappa_2}$.  We also obtain the following relations in $RO(C_2\times\Sigma_2)-$graded Bredon cohomology of $E_{\Sigma_2}C_2$: $c\ka_2=y_1x_3$ and $b\ka_2=x_1x_3$. Here $c$ and $b$ are the cohomology classes appearing in Theorem \ref{comp1}.
\end{remark}

\begin{remark} \label{cbg} We have that the motivic cohomology over complex numbers of $\mathbf{B} C _2$ (\cite{Voc}) is $$H^{*,*}(\mathbf{B}C _2,\Z/2)=H^{*,*}(\C,\Z/2)[s,t]/(s^2=\tau t)=\Z/2[s,t,\tau]/(s^2=\tau t).$$ Here $s$ is in bidegree $(1,1)$, $t$ is in bidegree $(2,1)$ and $\tau\in H^{0,1}(\C,\Z/2)$. According to Diagram \ref{TEST5}, in the complex case we can identify $s=\frac{\sigma\alpha}{u}$ (because $su=\sigma\alpha\in H^{-1+2\sigma}_{Br,C_2}(pt)$), $t=\frac{\sigma^2}{u}$ (because $tu=\sigma^2\in H^{2\sigma} _{Br,C_2}(pt)$), and $\tau=\frac{\alpha^2}{u}$ (because $u\tau=\alpha^2\in H^{-2+2\sigma}_{Br,C_2}(pt)$) as the unique generators of their respective groups. Here $\sigma$ is in degree $\sigma$, and $\alpha$ is in degree $\sigma-1$. Notice that in this case the relation $s^2=\tau t$ becomes obvious now. 

The complex realization of $\tau$ is $1$, so the relation $u\tau=\alpha^2$ becomes $Re(u)=\alpha^2$, which is obvious in the $RO(C_2)$-graded Bredon cohomology of $EC_2$.
\end{remark}
\begin{remark} \label{repr} The computation of the Borel motivic cohomology of the complex and real numbers can be used to independently compute the motivic cohomology of $\mathbf{B}C_2$ over the complex or real numbers. Using Theorem \ref{Borel} (which is independent of Voevodsky's computation of motivic cohomology of $\mathbf{B}C_2$) we identify the motivic cohomology ring of $\mathbf{B}C_2$ over the reals with $$(\oplus_{a\in \Z,b\geq 0} H^{a+b\sigma+b\sigma\otimes\epsilon}_{Br,K}(pt,\Z/2))[\kappa_2^{-1}],$$ and therefore it is given by the set of the following elements (with the obvious multiplication) : $$\sum \frac{x^{n_1}_1y^{m_1}_1x^{n_2}_3y^{m_2}_3}{\ka_2^b},$$ with $n_i,m_i\geq 0$ and $n_1+m_1=b=n_2+m_2, a=-m_1-m_2$. The non-trivial ring generators are chosen from those elements with $b=1$ as in Remark \ref{Vcomp}. The sums above give a different, but equivalent, set of $\Z/2$-vector space generators for the motivic cohomology ring of $\mathbf{B}C_2$.

In the complex case, the computation is simpler. The motivic cohomology ring of $\mathbf{B}C_2$ over the complex numbers can be identified with  $$\oplus_{a\in \Z,b\geq 0}H^{a+2b\sigma}_{Br,C_2}(pt,\Z/2)$$ and is therefore given by the elements
$$\sum \frac{\alpha^n\sigma^m}{u^{b}}$$
with $n+m=2b,n,m\geq 0$. The non-trivial ring generators are those elements (3 distinct elements with a relation among them) with $b=1, n+m=2, n,m\geq 0$. See Remark \ref{cbg}.
\end{remark}

\section{Bredon motivic cohomology of the real numbers}
In this section, we conclude that the Bredon motivic cohomology ring of the real numbers is a subring in the $RO(C_2\times \Sigma_2)$-graded Bredon cohomology of a point, and find a decomposition into $\MMt$-modules for it.

Let $$R:=\oplus_{b\geq 0,b+q\geq 0}H^{a+p\sigma,b+q\sigma}_{C_2}(\R,\Z/2)\hookrightarrow H^{\st,\st}_{C_2}(\R,\Z/2),$$
which is a cohomology subring on which the realization maps give an isomorphism (see the region $b,q\geq 0,b+q\geq 0$ of Figure \ref{TEST1}). 

We know from Theorem \ref{posisom} that
$R$ is ring isomorphic to a subring in the $RO(C_2\times \Sigma_2)$-graded Bredon cohomology of a point. We have that the subring
\begin{equation} \label{123}
R\simeq \oplus_{a,p\in \mathbb{Z},b\geq 0,b+q\geq 0} H^{a+p\sigma+b\epsilon+q\sigma\otimes\epsilon}_{Br,K}(pt,\Z/2)
\end{equation}
is given by a direct sum of four distinct pieces depending on the signs of $p$ and $q$. 

One piece that is contained above is simply the positive cone (see Theorem \ref{pcon}), which corresponds to the case where $p,q\geq 0$ (topological indexes) i.e.
$$\frac{\Z/2[x_1,y_1,x_2,y_2,x_3,y_3]}{(x_1y_2y_3+y_1x_2y_3+y_1y_2x_3)}.$$

Notice that $f=x_1y_2y_3+y_1x_2y_3+y_1y_2x_3=0$ is a trivial relation on $R$ because $R$ also contains the cohomological class $\ka_1\in H^{0,1+\sigma} _{C_2}(\R,\Z/2)$. This is because one can write, according to the relations in the $RO(C_2\times \Sigma_2)$-graded Bredon cohomology of a point (see \ref{Vcomp1} and \ref{Vcomp2}), that $$f=x_1\ka_1y_1+y_1x_2y_3+y_1y_2x_3=x_2y_1y_3+x_3y_2y_1+x_2y_1y_3+x_3y_1y_2=0.$$ The relations are also available in the subring $R$ because of the isomorphism \ref{123}. In conclusion, the subring $R$ contains the subring $\Z/2[x_1,y_1,x_2,y_2,x_3,y_3]$ modulo the relation $f=0$ and, as a subring, the motivic cohomology ring of $\R,$ which is $\Z/2[x_2,y_2]$ (see Proposition \ref{classic}). 

Besides these classes, other important topological cohomology classes belong to $R,$ including, for example, $\theta_1, \ka_1,\ka_3, \io_2$.

From Corollary \ref{EG} we know that $$M:=\oplus_{b<0, b+q\geq 0}H^{a+p\sigma,b+q\sigma} _{C_2}(\R,\Z/2)$$ is a $\MMt$-submodule in the $RO(C_2\times \Sigma_2)$-graded Bredon cohomology of a point with $\ka_2\in M$. Corollary \ref{aplic} shows that $$M\simeq \kappa_2 (\MMt[x_3,y_3,\kappa_2])\hookrightarrow H^{\st}_{Br,K}(pt,\Z/2).$$ 
Notice that this monomorphism is identity on all the generators $x_1,y_1,\frac{\theta_1}{x^n_1y^m_1}, x_3,y_3, \ka_2$ and therefore is identity on $M$ (here $\MMt$ is identified with $H^{\st,0}_{C_2}(\R,\Z/2)$).

It implies that the direct sum $P:=R\oplus M$ identifies with a subring in the $RO(C_2\times \Sigma_2)$-graded Bredon cohomology of a point. Obviously, from the previous section's discussion on motivic cohomology classes, $$\MMt[x_3,y_3]\subset R$$ and $$\ka_2\notin R.$$
Therefore we obtain the following description of $P$:
\begin{proposition} We have the following isomorphism of rings  $$P\simeq (R,\kappa_2)\subset H^\st _{Br,K}(pt,\Z/2),$$
where $(R,\kappa_2)$ is the subring in $H^\st _{Br,K}(pt,\Z/2)$ generated by $R$ and $\kappa_2$. 
\end{proposition}
We notice that $\theta_1,\io_3,\ka_i\in P$ are nontrivial cohomology classes and $\theta_2,\theta_3\notin P$. For a review of these cohomological classes see Section 2.2.

We will discuss next the region $b+q<0$, $b\geq 0$ from Figure \ref{TEST1}.

Let $$NC:=\oplus _{b\geq 0, b+q<0}H^{\st,b+q\sigma} _{C_2}(\R,\Z/2),$$ which is a $\MMt$-submodule of the Bredon motivic cohomology of real numbers (see Corollary \ref{itilde} and Theorem \ref{ncinj}), and also fits into an inclusion of multiplicative maps $$NC\hookrightarrow \rH^{\st,\st} _{C_2}(\wt\EG C_2,\Z/2)\hookrightarrow \rH^{\st}_{Br,K}(\tilde{E}_{\Sigma_2}C_2,\Z/2).$$ We have that the quotient map $$\tilde{E}_{\Sigma_2}C_2\rightarrow \tilde{E}_{\Sigma_2}C_2/C_2\stackrel{htpy}{\simeq} \Sigma B_{\Sigma_2}C_2$$
induces an isomorphism of (non-unital) commutative rings $$\rH^{*+*\epsilon}_{Br,C_2}(\tilde{E}_{\Sigma_2}C_2/C_2)\simeq \rH^{*+*\epsilon}_{Br,C_2}(\Sigma B_{\Sigma_2}C_2).$$ Therefore, $NC$ has zero products, because the above suspension gives the commutative ring $$\rH^{\st}_{Br,C_2}(\tilde{E}_{\Sigma_2}C_2)\simeq\rH^{*+*\epsilon}_{Br,C_2}(\Sigma B_{\Sigma_2}C_2)[x_1^{\pm{1}},x_3^{\pm{1}}]$$
with zero multiplication. 

The image of $NC$ in $\rH^{\st}_{Br,K}(\tilde{E}_{\Sigma_2}C_2,\Z/2)$ is $$NC\subset\oplus _{2\leq a\leq 2b+1}\rH^{a-b+b\epsilon} _{Br,C_2}(\Sigma B_{\Sigma_2}C_2,\Z/2)[x_1^{\pm 1},x^{\pm 1}_3]\simeq \oplus _{1-b\leq a\leq b}\rH^{a+b\epsilon} _{Br,C_2}(B_{\Sigma_2}C_2,\Z/2)[x_1^{\pm 1},x^{\pm 1}_3].$$
This implies, using the $RO(\Sigma_2)-$graded Bredon cohomology of $B_{\Sigma_2}C_2$ from Theorem \ref{comp1}, that $$NC\subset \{x^n_2y^m_2\Sigma (b^pc),x^n_2y^m_2\Sigma (b^{p+1})\}[x^{\pm 1}_1,x^{-1}_3],$$
the subset with the degree of $x^{-1}_3$ given by $q\leq -b\leq 0,$ where $b$ is the degree of $\epsilon$ and $p\geq 0$ i.e.
 $$NC=\oplus _{a\leq 2b+1} x^{-b-1}_3\rH^{a+b\epsilon}_{Br,C_2}(\Sigma B_{\Sigma_2}C_2,\Z/2)[x^{\pm 1}_1,x^{-1}_3]=$$
 $$=\{x^{-n-m-p-2}_3x^n_2y^m_2\Sigma (b^pc),x^{-n-m-p-1}_3x^n_2y^m_2\Sigma (b^{p+1})\}[x^{\pm 1}_1,x^{-1}_3].$$
Moreover, for $a\leq 2b+1$, $b\geq 0$, $b+q<0$ we have that
  \[
\xymatrixrowsep{0.3in}
\xymatrixcolsep{0.25in}
\xymatrix@-0.9pc{
\rH^{a+p\sigma,b+q\sigma} _{C _2}(\wt \EG C _2,\Z/2)\ar[r]^{\simeq}\ar[d]^{\simeq}&
H^{a+p\sigma,b+q\sigma} _{C _2}(\R,\Z/2)\ar[d]
\\ 
\rH^{a-b+(p-q)\sigma+b\epsilon+q\sigma\otimes\epsilon} _{Br,K}(\tilde{E} _{\Sigma _2}C _2,\Z/2)\ar[r]&
H^{a-b+(p-q)\sigma+b\epsilon+q\sigma\otimes\epsilon} _{Br,K}(pt,\Z/2)
}
\]
and from Proposition \ref{ncinj} (see also Proposition \ref{ncinj2}) we obtain an injective multiplicative map $$NC\hookrightarrow H^{\st} _{Br,K}(pt,\Z/2),$$ with zero products in the domain.

 $NC$ has a $\MMt$-module structure given by $$\frac{\theta_1}{x_1^{n_1}y_1^{m_1}}\alpha=0;$$ and $$y_1\frac{x^n_2y^m_2\Sigma (b^pc)}{x^{n+m+p+2}_3}=x_1\frac{x^{n+1}_2y_2^m\Sigma cb^{p-1}}{x_3^{n+m+p+2}}+x_1\frac{y_2^{m+1}x_2^n\Sigma b^p}{x_3^{n+m+p+2}},p\geq 1;$$ and 
$$y_1\frac{x^n_2y^m_2\Sigma b^p}{x^{n+m+p+1}_3}=x_1\frac{x_2^{n}y_2^{m}\Sigma cb^{p-1}}{x^{n+m+p+1}_3},p\geq 1;$$ and
$$y_1\frac{x^n_2y^m_2\Sigma c}{x^{n+m+1}_3}=0.$$

We can give now a description of the Bredon motivic cohomology ring of real numbers in the following theorem:
\begin{theorem}\label{ft} We have a decomposition of  $\MMt$-modules included through a ring map in $H^{\st}_{Br,K}(pt,\Z/2)$
$$H^{\st,\st} _{C_2}(\R,\Z/2)=(R,\kappa_2)\oplus NC\hookrightarrow H^{\st}_{Br,K}(pt,\Z/2)$$
with $\kappa_2$ in degree $(2\sigma-2,\sigma-1)$, and $(R,\kappa_2)\hookrightarrow H^\st_{Br,K}(pt,\Z/2)$ a subring, and $NC$ is a $\MMt$-submodule with zero products. 
\end{theorem} 
\begin{proof} The first $\MMt-$module in the statement's direct sum corresponds to $b+q\geq 0$ and the second corresponds to $b+q<0$. The structure of $\MMt-$module is given by the fact that multiplication by elements in $\oplus H^{a+p\sigma,0} _{C_2}(\R,\Z/2)\simeq \oplus H^{a+p\sigma}_{Br,C_2}(pt,\Z/2)$ doesn't change the cohomological weight. The ring structure is induced from the ring structure of $RO(C_2\times\Sigma_2)-$graded Bredon cohomology ring of a point via the realization which is injective. From the above discussion, $NC$ has zero products and it is injectively embedded in $RO(C_2\times \Sigma_2)-$graded Bredon cohomology of a point. It is naturally a $\MMt-$module because $NC$ has weight $b+q<0$ and multiplication by elements of $\MMt$ don't change the weight.
\end{proof}
The image of $NC$ in $H^{\st}_{Br,K}(pt,\Z/2)$ is more complicated and was determined in (\cite{DV}, Theorem 5.0.6). As we saw above the image of $(R,\kappa_2)$ in $H^{\st}_{Br,K}(pt,\Z/2)$ is identity. In \cite{DV} we determine $H^{\st,\st} _{C_2}(\R,\Z/2)$ completely in terms of cohomology classes and relations as a subring in $H^{\st}_{Br,K}(pt,\Z/2)$. For example we prove there that $(R,\kappa_2)$ can contain nilpotents and multiplication between an element in $NC$ and and element in $(R,\kappa_2)$ can produce a nontrivial nilpotent in $(R,\kappa_2)$. 

Notice that the negative cone of the $RO(C_2\times \Sigma_2)$-graded Bredon cohomology ring of a point  (by definition $b,p,q<0$, see Section 2.1) is completely outside the image of the realization map from Theorem \ref{ft} because of the vanishing range of Bredon motivic cohomology given in Proposition \ref{vanr}. 

Also, according to Remark \ref{tcc} and Remark \ref{ncc}, there are other topological classes outside the negative cone that are not in the image of this realization map. Moreover the topological cohomological classes $\io_1$  and $\theta_2$ are not in the image of the realization map and do not belong to the negative cone or to the examples in Remarks \ref{tcc} and \ref{ncc}.

We obtain from Theorem \ref{ft} or Theorem \ref{prop:ptiso} the following interesting subrings of $(R,\kappa_2)$:
\begin{corollary} We have that $$H^{*+*\sigma,*}_{C_2}(\R,\Z/2)\simeq \MMt[x_2,y_2],$$
 $$H^{*+*\sigma,*\sigma}_{C_2}(\R,\Z/2)\simeq \MMt[x_3,y_3].$$
\end{corollary}
\subsection{Real closed fields}
\begin{theorem}\label{rc} Let $k$ be a real closed field. Then
$$H^{\st,\st} _{C_2}(\R,\Z/2)\simeq H^{\st,\st} _{C_2}(k,\Z/2),$$
$$H^{\st,\st} _{C_2}({\EG C_2},\Z/2)\simeq H^{\st,\st} _{C_2}({\EG C_2}_k,\Z/2).$$
 \end{theorem}
\begin{proof} Bredon motivic cohomology $$F(U)=H^{a+p\sigma,b+q\sigma}_{C_2}(X\times U)$$ is a homotopy invariant presheaf with equivariant transfers for any $X$ a smooth $C_2$-scheme (\cite{HOV}, \cite{HOV1}) over a field $k$ of characteristic zero. In particular, the restriction of $F$ to $Sm/k$ is a pseudo pretheory on $Sm/k$. From the main result of $\cite{RO}$ (with the comment of \cite{HO}, Theorem 4.18) we know that if $k\subset \R$ is a real closed subfield then $$F(U)\simeq F(U_\R)$$ for any smooth $C_2$-scheme $U$ over $k,$ implying the result in the theorem. 
Let $k$ be an arbitrary real closed field and $L=k[i]$ the corresponding algebraic closed field. We can write $L=\cup L_\alpha$ with $L_\alpha$ a subfield of $L$ of finite transcendence degree over $\mathbb{Q}$ and $\alpha\in A$ a well-order set and $k=\cup k_\alpha$ where $k_\alpha=L_\alpha\cap k$. In Theorem 2.20 \cite{HO} it is proved that $k_\alpha$ is isomorphic to a real closed field embedded in $\R$. Using the above considerations, we conclude that for a $C_2$-smooth scheme $X$ we have 
$$H^{a+p\sigma,b+q\sigma}_{C_2}(X_k)=colim_\alpha H^{a+p\sigma,b+q\sigma}_{C_2}(X_{k_\alpha})=H^{a+p\sigma,b+q\sigma}_{C_2}(X_\R).$$ 
\end{proof}
For a real closed field $k\subset \R$ and $X$ a $C_2$-smooth scheme over $k$ there is a cycle map  
$$cyc_k: H^{a+p\sigma,b+q\sigma}_{C_2}(X,\Z/2)\rightarrow  H^{a+p\sigma,b+q\sigma}(X_\R,\Z/2)\rightarrow H^{a-b+(p-q)\sigma+b\epsilon+q\sigma\otimes\epsilon}_{Br,K}(X(\C),\Z/2).$$
Therefore, according to the proof of Theorem \ref{rc} and Theorem \ref{ft}, we conclude that the Bredon motivic cohomology ring of a real closed field embedded in $\R$ is, up to isomorphism, a nontrivial proper subring in the $RO(C_2\times \Sigma_2)$-graded Bredon cohomology ring of a point. 

We also conclude that the Borel motivic cohomology ring of a real closed field embedded in $\R$ is a nontrivial subring in the $RO(C_2\times \Sigma_2)$-graded Bredon cohomology ring of $E_{\Sigma_2}C_2$. 

In the case of a $C_2$-smooth scheme over a real closed field $k\subset \R$ the range of isomorphism for $cyc_k$ is in general much more restricted. As a generalization of [\cite{HO}, Theorem 4.18] and of [\cite{HV}, Corollary 5.13] and according to [\cite{HOV1}, Theorem 7.10] we obtain the following:
\begin{corollary} Let $X$ be a smooth  scheme over a real closed field $k$ embedded in $\R$. Then the cycle map $cyc_k$ 
$$cyc_k: H^{a+p\sigma,b+q\sigma}_{C_2}(X,\Z/2)\rightarrow H^{a-b+(p-q)\sigma+b\epsilon+q\sigma\otimes\epsilon}_{Br,K}(X(\C),\Z/2).$$
is an isomorphism if $a+p\leq b+q$ and $a\leq min\{b-q,b\}$ and a monomorphism if $a+p\leq b+q+1$ and $a\leq min\{b-q,b\}+1$. 
\end{corollary}
\bibliographystyle{plain}
 \bibliography{road1}

\vspace{10pt}
	\noindent
	Bill Deng\\
	UNSW Sydney\\
	Department of Mathematics and Statistics\\
	NSW 2052 Australia\\
	\texttt{bill.deng@student.unsw.edu.au}
	
\vspace{10pt}
	\noindent
	Mircea Voineagu\\
	UNSW Sydney\\
	Department of Mathematics and Statistics\\
	NSW 2052 Australia\\
	\texttt{m.voineagu@unsw.edu.au}
\end{document}